\documentclass[11pt]{article}
\usepackage{graphicx} % Required for inserting images
\usepackage{authblk}
\usepackage{amsthm}
\usepackage{amsmath}
\usepackage{amssymb}
\usepackage{enumitem}
\usepackage{mathrsfs}
\usepackage[colorlinks,linkcolor=blue,citecolor=blue,urlcolor=blue,hypertexnames=true,backref=page]{hyperref}
\usepackage[margin=1in]{geometry}

\theoremstyle{definition}
\newtheorem{thm}{Theorem}[section]
\newtheorem{lem}[thm]{Lemma}
\newtheorem{cor}[thm]{Corollary}
\newtheorem{dfn}[thm]{Definition}
\newtheorem{rem}[thm]{Remark}

\newcommand{\lra}{\longrightarrow}
\newcommand{\mb}{\textbf{mbC}}

\newcommand{\ov}{\overline}
\newcommand{\wh}{\widehat}
\newcommand{\vd}{\mathbf{vD}}
\newcommand{\vdvd}{\vdash_{\vd}}
\newcommand{\modvd}{\models_{\vd}}
\newcommand{\lang}{\mathcal{L}}
\newcommand{\logl}{\mathscr{L}}
\newcommand{\Int}{\mathrm{Int}}

\newcommand{\pow}{\mathcal{P}}

\title{Topological semantics for a non-self-extensional LFI}
\author[1]{Esha Jain}
\author[1]{Sankha S. Basu}
 \affil[1]{Department of Mathematics\\
  Indraprastha Institute of Information Technology-Delhi\\
  New Delhi, India.}
  \date{November 6, 2025}

\begin{document}

\maketitle

\begin{abstract}
In this article, we have introduced a Logic of Formal Inconsistency (LFI) that we call $\vd$. This logic is non-self-extensional, i.e., the replacement property, or the rule for substitution of equivalents, does not hold. A Hilbert-style presentation for the logic has been provided. Then, a topological semantics for $\vd$ has been described, subsequent to which we have established the Soundness and Completeness results for it with respect to this semantics.
\end{abstract}


\tableofcontents

\section{Introduction}
The concept of topological semantics for logics is not new (see \cite{Goodman1981,McKinseyTarski1944}). There are several articles that illustrate the relationship between topological properties and logical theories. For example, the relationship between the various separation principles in topology and logical theories is established in \cite{mortensen2000topological}, which is then extended by defining \emph{connected formulas} that are used to establish a link between connected topological spaces and logical theories in \cite{baskent2013}. 

The above-mentioned articles, along with \cite{ConiglioPrieto-Sanabria2017} and \cite{fuenmayor2023semantical}, exemplify the use of topological semantics for \emph{paraconsistent} logics. A logic $\mathscr{L}=\langle\lang,\vdash_{\mathscr{L}}\rangle$ is said to be \emph{paraconsistent} if it invalidates a \emph{principle of logical explosion}. The most common principle of logical explosion is expressed as follows: for all $\alpha,\beta\in\lang$, $\{\alpha,\neg\alpha\}\vdash_{\mathscr{L}}\beta$. This is called \emph{ECQ (ex contradictione sequitor quodlibet)}. $\mathscr{L}$ is said to be paraconsistent (with respect to $\neg$), if ECQ fails in it, i.e., there exist $\alpha,\beta\in\lang$ such that
$\{\alpha,\neg\alpha\}\not\vdash_{\mathscr{L}}\beta$. In other words, negation inconsistency does not lead to absolute inconsistency, or triviality, in such a logic. There are other principles of explosion (see \cite{BasuRoy2024}), which can produce other kinds of paraconsistency. 

Topological semantics for a special class of paraconsistent logics is presented in a recent article \cite{fuenmayor2023semantical}. These paraconsistent logics belong to the class of Logics of Formal Inconsistency (LFIs) that are \emph{self-extensional} in the sense of \cite{Wojcicki1988}. LFIs (see Definition \ref{def:lfi}) are paraconsistent logics in which the explosion law is allowed in a `local' or `controlled' way, as mentioned in \cite{CarnielliConiglio2016}. This is done by having an object-level denotation of consistency via a unary `consistency operator' in the logic.

A \emph{self-extensional} logic is one that obeys the \emph{replacement property} (see Definition \ref{def:rep'ment}). An example of an LFI satisfying the replacement property is $\mathbf{RmbC}$, introduced in \cite{CarnielliConiglioFuenmayor2022}. This is an extension of the well-known LFI $\mathbf{mbC}$, that is not self-extensional. This is discussed in \cite{Carnielli2007} and \cite{CarnielliConiglio2016}. In $\mathbf{mbC}$, the unary connectives $\neg$ and $\circ$ are partially determined, giving rise to non-deterministic semantics and this is what leads to the failure of the replacement property. 

In this paper, we present a new LFI, which we have named as $\vd$, that is not self-extensional and provide a topological semantics for this logic. A distinguishing feature of  $\vd$ is that the connective $\lor$ (disjunction) does not obey the rules for classical disjunction. This is unlike other LFIs. More precisely, $(\alpha\lra\gamma)\lra((\beta\lra\gamma)\lra((\alpha\lor\beta)\lra\gamma))$ is not a theorem of $\vd$. This plays a crucial role in the failure of the replacement property in this logic. Therefore, \textbf{d}isjunction is \textbf{v}ague in this logic.

The article is structured as follows. We introduce the logic $\vd$ in Section \ref{sec:logic} by a Hilbert style presentation. Topological semantics for this logic is described in Section \ref{sec:top}. This is followed by the Soundness theorem for $\vd$ with respect to the topological semantics in Section \ref{sec:soundness}. Section \ref{sec:lfi} shows that $\vd$ is indeed an LFI and is non-self-extensional. Finally, in Section \ref{sec:com'ness}, we establish that $\vd$ is complete with respect to the topological semantics.

\section{The logic \texorpdfstring{$\vd$}{vd}}\label{sec:logic}

Let $\mathcal{L}$ be the set of formulas generated inductively over a denumerable set of variables $V$ using a finite set of connectives or operators, called the \emph{signature}. A logic, with signature $\Sigma$, is then a pair $\mathcal{\mathscr{L}}=\langle\mathcal{L},\vdash_{\mathscr{L}}\rangle$, where $\vdash_{\mathscr{L}}\,\subseteq\pow(\mathcal{L})\times\mathcal{L}$ is the consequence relation of the logic $\mathscr{L}$. The subscript $\logl$ is dropped from $\vdash_\logl$ when there is no ambiguity regarding the logic under discussion.

\begin{dfn}
A logic $\logl=\langle\lang,\vdash\rangle$ is called \emph{Tarskian} if it satisfies the following properties. For every $\Gamma\cup\Delta\cup\{\alpha\}\subseteq\lang$,
    \begin{enumerate}[label=(\roman*)]
        \item if $\alpha\in\Gamma$, then $\Gamma\vdash\alpha$ (\emph{Reflexivity});
        \item if $\Gamma\vdash\alpha$ and $\Gamma\subseteq\Delta$, then $\Delta\vdash\alpha$ (\emph{Monotonicity});
        \item if $\Delta\vdash\alpha$ and $\Gamma\vdash\beta$ for all $\beta\in\Delta$, then $\Gamma\vdash\alpha$ (\emph{Transitivity/ Cut}).
    \end{enumerate}
    $\logl$ is said to be \emph{finitary} if for every $\Gamma\cup\{\alpha\}\subseteq\lang$, $\Gamma\vdash\alpha$ implies that there exists a finite $\Gamma_0\subseteq\Gamma$ such that $\Gamma_0\vdash\alpha$.

    $\logl$ is said to be \emph{structural} if for every $\Gamma\cup\{\alpha\}\subseteq\lang$, $\Gamma\vdash\alpha$ implies that $\sigma[\Gamma]\vdash\sigma(\alpha)$, for every substitution $\sigma$ of formulas for variables.

    The logic $\logl$ is called \emph{standard} if it is Tarskian, finitary and structural.
\end{dfn}

\begin{rem}
    The set of formulas $\lang$ can also be described as the formula algebra over $V$ of some type/ signature. The formula algebra has the universal mapping property for the class of all algebras of the same type as $\lang$ over $V$, i.e., any function $f:V\to A$, where $A$ is the universe of an algebra $\mathbf{A}$ of the same type as $\mathcal{L}$, can be uniquely extended to a homomorphism from $\lang$ to $\mathbf{A}$ (see \cite{FontJansanaPigozzi2003,Font2016} for more details).
    
    A \emph{substitution} can then be defined as any function $\sigma:V\to\lang$ that extends to a unique endomorphism (also denoted by $\sigma$) from $\lang$ to itself via the universal mapping property. The logic $\logl$ is then defined to be structural as above.
\end{rem}

The logic $\vd=\langle\lang,\vdvd\rangle$ is a logic with signature $\Sigma=\{\land,\lor,\lra,\neg,\circ\}$, where $\neg$ and $\circ$ are unary, and the rest are binary operators, and is induced by the following Hilbert-style presentation.
 
 \textsc{Axiom Schema:}\begin{enumerate}[label=(\arabic*)]
     \item \label{ax:i}$\alpha\lra(\beta\lra\alpha)$
     \item \label{ax:ii}$(\alpha\lra(\beta\lra\gamma))\lra ((\alpha\lra\beta)\lra(\alpha\lra\gamma))$
     \item \label{ax:iii}$\alpha\lra(\beta\lra(\alpha\land\beta))$
     \item \label{ax:iv}$(\alpha\land\beta)\lra \alpha$
     \item \label{ax:v}$(\alpha\land\beta)\lra \beta$
     \item\label{ax:vi} $\alpha\lra(\alpha\lor\beta)$
     \item \label{ax:vii}$\beta\lra(\alpha\lor\beta)$
     \item\label{ax:viii}$(\alpha\lra\beta)\lor\alpha$
     \item \label{ax:ix}$\alpha\lor\neg\alpha$
     \item \label{ax:x}$(\alpha\lra\gamma)\lra((\neg\alpha\lra\gamma)\lra((\alpha\lor\neg\alpha)\lra\gamma))$
     \item \label{ax:xi}$((\alpha\lra\beta)\lra\gamma)\lra((\alpha\lra\gamma)\lra(((\alpha\lra\beta)\lor\alpha)\lra\gamma))$
     
     \item \label{ax:bc1}$\circ\alpha\lra(\alpha\lra(\neg\alpha\lra\beta))$

     \item \label{ax:ciw}$\circ\alpha\lor(\alpha\land\neg\alpha)$
     \item \label{ax:xiv}$(\circ\alpha\lra\gamma)\lra(((\alpha\land\neg\alpha)\lra\gamma)\lra((\circ\alpha\lor(\alpha\land\neg\alpha))\lra\gamma))$
\end{enumerate}

Using a similar method as in \cite{CarnielliConiglio2016}, a bottom formula can be defined as follows.

For any formula $\beta$, a bottom formula $\bot:=\beta \land (\neg\beta\land \circ \beta)$. That this is indeed a bottom particle can be justified by using axiom \ref{ax:bc1}. Then, using $\bot$ and $\lra$, we can define a unary operator $\sim$, which behaves as a classical negation in $\vd$. Then we have the following axioms involving $\sim$.

\begin{enumerate}[resume, label=(\arabic*)] 
     \item \label{ax:xv}$\sim\neg\alpha\lra\sim\neg\sim\neg\alpha$
     \item \label{ax:xvi}$\sim\neg\sim\neg\alpha\lra\sim\neg\alpha$
     \item \label{ax:xvii}$\sim\neg(\alpha\land\beta)\lra\sim\neg\alpha\land\sim\neg\beta$
     \item \label{ax:xviii}$\sim\neg\alpha\land\sim\neg\beta\lra\sim\neg(\alpha\land\beta)$
\end{enumerate}

 \textsc{Inference Rules:}
 \begin{enumerate}[label=(\arabic*)]
     \item \label{rule:mp} $\begin{array}{c}
        \alpha\qquad\alpha\lra\beta\\
        \hline
        \beta
     \end{array}\;$(MP)
     
     \item \label{rule:neg}$\begin{array}{c}
         \alpha\\
         \hline
         \neg\alpha\lra\sim\alpha
     \end{array},\;\hbox{provided $\alpha$ is a theorem (see Definition \ref{def:derivation})}$. 
\end{enumerate}

The fact that $\vd$ is induced by a Hilbert-style presentation, i.e., it is a Hilbert-style logic is captured in the following definitions of \emph{syntactic derivation} and \emph{syntactic entailment} in $\vd$.

\begin{rem}
    The axioms \ref{ax:i}-\ref{ax:viii} in the above Hilbert-style presentation of $\vd$ are theorems of positive classical propositional logic (CPL$^+$) and all extensions of the LFI $\mb$ as discussed in \cite{CarnielliConiglio2016}. The axioms \ref{ax:bc1} and \ref{ax:ciw} are called bc1 and ciw, respectively, in \cite{CarnielliConiglio2016}. The axiom \ref{ax:bc1} is also called the \emph{gentle explosion law}.
\end{rem}

\begin{dfn}\label{def:derivation}
Let $\Gamma\cup\{\varphi\}\subseteq\lang$. A \emph{derivation} of $\varphi$ from $\Gamma$ in $\vd$ is a finite sequence $(\varphi_{1},\ldots,\varphi_{n})$ of elements in $\lang$, where $\varphi_{n}=\varphi$ and for each $1\le i\le n$, 

\begin{enumerate}[label=(\roman*)]
    \item $\varphi_{i}$ is an instance of an axiom of $\vd$, or
    \item $\varphi_{i}\in \Gamma$, or
    \item there exist $1\le j,k<i$ such that $\varphi_{i}$ is obtained from $\varphi_{j},\varphi_{k}$ by MP.
    \item there exists  $j<i$ such that $\varphi_j$ is a \emph{theorem} and $\varphi_{i}$ is the result of an application of the second inference rule on $\varphi_j$.
\end{enumerate}
We say that $\varphi$ is \emph{syntactically derivable} or \emph{syntactically entailed} from $\Gamma$, and write $\Gamma\vdvd\varphi$, if there is a derivation of $\varphi$ from $\Gamma$. If $\Gamma=\emptyset$, then $\varphi$ is called \textit{theorem}, i.e., $\varphi$ is a theorem if it can be derived using only axioms and rules.
\end{dfn}

\begin{rem}
    The clause ``$\alpha$ is a theorem'' is needed in the second rule of inference (Rule \ref{rule:neg}) in the above Hilbert system, since otherwise, we have the following derivation from $\{\alpha,\neg\alpha\}$, where $\alpha\in\lang$ is any formula.
    \[
    \begin{array}{rlll}
         \{\alpha,\neg\alpha\}\vdvd&1.&\alpha&[\hbox{Hypothesis}]\\
         &2.&\neg\alpha\lra\sim\alpha&[\hbox{Unrestricted use of Rule \ref{rule:neg}}]\\
         &3.&\neg\alpha&[\hbox{Hypothesis}]\\
         &4.&\sim\alpha&[\hbox{MP on (2) and (3)}]\\
         &5.&\alpha\lra\bot&[\hbox{Definition of $\sim$}]\\
         &6.&\bot&[\hbox{MP on (1) and (5)}]\\
         &7.&\beta\land(\neg\beta\land\circ\beta)&[\hbox{Definition of $\bot$; $\beta$: any formula}]\\
         &8.&(\beta\land(\neg\beta\land\circ\beta))\lra\beta&[\hbox{Axiom \ref{ax:iv}}]\\
         &9.&\beta&[\hbox{MP on (7) and (8)}]
    \end{array}
    \]
    Thus, $\{\alpha,\neg\alpha\}\vdvd\beta$ for any $\alpha,\beta\in\lang$. In other words, ECQ holds, with respect to $\neg$, in $\vd$. This, however, cannot be allowed. Hence, Rule \ref{rule:neg} can only be applied on an $\alpha$ that is a theorem. The above arguments also point out that $\{\alpha,\neg\alpha\}$ explodes, if $\alpha$ is a theorem.
\end{rem}

\begin{rem}\label{rem:vd-Tarskian/finitary}
    It follows from the above definition of a syntactic consequence in $\vd$, that $\vd$ is a Tarskian, finitary and structural logic. 
\end{rem}

\begin{thm}\label{thm:DedThm}
    The Deduction theorem holds in $\vd=\langle\lang,\vdvd\rangle$, i.e., for any $\Gamma\cup\{\alpha,\beta\}\subseteq\lang$, $\Gamma\cup\{\alpha\}\vdvd\beta$ iff $\Gamma\vdvd\alpha\lra\beta$.
\end{thm}

\begin{proof}
This is standard in the presence of axioms \ref{ax:i} and \ref{ax:ii} and MP as a rule of inference.
\end{proof}

The next theorem shows that $\sim$ is explosive  and the law of excluded middle in terms of $\sim$ holds in $\vd$.

\begin{thm}\label{thm:prop_cl_neg}
  The following are theorems of $\vd$. 
  \begin{enumerate}[label=(\roman*)]
    \item \label{ax:expl} $\vdvd\alpha\lra(\sim\alpha\lra\beta)$.
    \item \label{ax:lem}
    $\vdvd\sim\alpha\lor\alpha$
    \end{enumerate}
\end{thm}    

\begin{proof}
\begin{enumerate}[label=(\roman*)]
    \item 
    We can  construct the following derivation of $\beta$ from $\{\alpha,\sim\alpha\}$.
    \[
    \begin{array}{rll}
    1.&\alpha&[\hbox{Hypothesis}]\\
    2.&\sim\alpha&[\hbox{Hypothesis}]\\
    3.&\alpha\lra\bot&[\hbox{Definition of $\sim$}]\\
    4.&\bot&[\hbox{MP on (1) and (2)}]\\
    5.&\alpha\land(\neg\alpha\land\circ\alpha)&[\hbox{Definition of $\bot$}]\\
    6.&(\alpha\land(\neg\alpha\land\circ\alpha))\lra(\neg\alpha\land\circ\alpha)&[\hbox{Axiom \ref{ax:v}}]\\
    7.&(\neg\alpha\land\circ\alpha)&[\hbox{MP on (5) and (6)}]\\
    8.&(\neg\alpha\land\circ\alpha)\lra\neg\alpha& [\hbox{Axiom \ref{ax:iv}}]\\
    9.&\neg\alpha&[\hbox{MP on (7) and (8)}]\\
    10. &(\neg\alpha\land\circ\alpha)\lra\circ\alpha& [\hbox{Axiom \ref{ax:v}}]\\
    11.&\circ\alpha&[\hbox{MP on (7) and (10)}]\\ 12.&\alpha\lra(\neg\alpha\lra(\circ\alpha\lra\beta))&[\hbox{Axiom \ref{ax:bc1}}]\\
    13.&\neg\alpha\lra(\circ\alpha\lra\beta)&[\hbox{MP on (1) and (12)}]\\
    14.&\circ\alpha\lra\beta&[\hbox{MP on (9) and (13)}]\\
    15.&\beta&[\hbox{MP on (11) and (14)})
     \end{array}
     \]
Thus, $\{\alpha,\sim\alpha\}\vdvd\beta$.
Hence, by applying the Deduction theorem (Theorem \ref{thm:DedThm}) twice, we have $\vdvd\alpha\lra(\sim\alpha\lra\beta)$.
 
\item By Axiom \ref{ax:viii}, $\vdvd(\alpha\lra\bot)\lor\alpha$. Then, by definition of $\sim$, we have 
$\vdvd\sim\alpha\lor\alpha$.
\end{enumerate}
\end{proof}
    
\begin{rem}
    The above theorem shows that $\sim$ indeed behaves like Boolean negation in $\vd$. 

    The law of excluded middle holds for both negations - $\neg$ and $\sim$.
\end{rem}
 
\section{Topological semantics}\label{sec:top}
In this section, we describe a topological semantics for the logic $\vd=\langle\lang,\vdvd\rangle$ described in the previous section. Given a topological space $\langle X,\tau\rangle$ and $P\subseteq X$, we will denote the interior, closure, and complement of $P$ by $\Int(P),\; \ov{P}$, and $P^{c}$, respectively.
  \begin{dfn}\label{def:val'on}
    A \emph{topological structure}  for $\vd$ is a topological space $\mathcal{T}=\langle X,\tau\rangle$. A \emph{topological model} for $\vd$ is a pair $\mathcal{M}=\langle\mathcal{T},v\rangle$, where $\mathcal{T}=\langle X,\tau\rangle$ is a topological structure for $\vd$ and $v:\lang\lra\pow(X)$ is a function, called a \emph{valuation}, that satisfies the following conditions.
    \begin{enumerate}[label=(\arabic*)]
        \item $v(\alpha\lra\beta)=v(\alpha)^{c}\cup v(\beta)$;
        \item $v(\neg\alpha)=\overline{v(\alpha)^{c}}$;
        \item $v(\sim \alpha)=v(\alpha)^{c}$;
        
        \item $v(\alpha\land\beta)=v(\alpha)\cap v(\beta)$;
        \item $v(\alpha\lor\beta)$ is such that $v(\alpha)\cup v(\beta)\subseteq v(\alpha \lor \beta)$; 
        \item $v(\circ \alpha)=v(\alpha)^{c}\cup\mathrm{Int}(v(\alpha))$.
    \end{enumerate}
\end{dfn}

\begin{rem}
    The definition of $v$ for the connectives $\neg$, $\lra$ and $\land$ in $\vd$ are the same as for the logic $\mathbf{LTop}$ given in \cite{ConiglioPrieto-Sanabria2017}. The Boolean negation $\sim$ is also interpreted as expected. We, however, differ in the semantic condition for $\lor$. 
    
    Perhaps, it is intuitively natural to have $v(\circ \alpha)=\mathrm{Int}(v(\alpha)^{c})\cup\mathrm{Int}(v(\alpha))$, which is the complement of the boundary of $v(\alpha)$. This is the case for the logic $\mathbf{LTop}$ in \cite{ConiglioPrieto-Sanabria2017}. We, however, differ from that here and have $v(\circ\alpha)=v(\alpha)^{c}\cup\mathrm{Int}(v(\alpha))$. This can be traced back to the suggested interpretation of the operator in \cite{marcos2005}.  One could also try other possible valuations.
\end{rem}

\begin{rem}
    It is clear from the above definition that the operators $\lra,\;\neg,\;\sim,\;\land,\;\circ$ are fully determined. The operator $\lor$, on the other hand, is only partially determined. Suppose $\mathcal{M}=\langle\mathcal{T},v\rangle$ is a topological model for $\vd$, where $\mathcal{T}=\langle X,\tau\rangle$ is a topological structure for $\vd$ and $v$ is a valuation as defined above. Then, for any $\alpha,\beta\in\lang$, $v(\alpha\lor\beta)$ cannot be determined from $v(\alpha),v(\beta)$ unless $v(\alpha)\cup v(\beta)=X$. $v(\alpha\lor\beta)$ can take any value such that $v(\alpha)\cup v(\beta)\subseteq v(\alpha\lor\beta)\subseteq X$.
\end{rem}

We now define the semantic consequence relation for $\vd$ in terms of the above topological interpretation.

\begin{dfn}\label{def:sem'tic truth}
Suppose $\Gamma\cup\{\alpha\}\subseteq\lang$.

\begin{enumerate}[label=(\alph*)]
    \item $\alpha$ is said to be \emph{true} in a topological model  $\mathcal{M}=\langle\mathcal{T},v\rangle$ for $\vd$, where $\mathcal{T}=\langle X,\tau\rangle$ is a topological structure for $\vd$ and $v$ is a valuation, if $v(\alpha)=X$. This is written as $\mathcal{M}\modvd\alpha$.
    \item $\alpha$ is said to be \emph{valid}, if  $\mathcal{M}\modvd\alpha $, for every topological model $\mathcal{M}$ for $\vd$. This is denoted by $\modvd\alpha$.
    \item $\alpha$ is said to be a \emph{semantical consequence} of $\Gamma$ in $\vd$, if either $\modvd\alpha$, or there exists a finite non-empty $\Gamma_{0}\subseteq\Gamma$ such that $\displaystyle\bigcap_{\gamma\in \Gamma_{0}}v(\gamma)\subseteq v(\alpha)$, in every topological model $\mathcal{M}=\langle\mathcal{T},v\rangle$ for $\vd$. This is written as $\Gamma\modvd\alpha$.
\end{enumerate}
\end{dfn}

The following lemma provides a convenient alternative method for establishing the truth of a conditional in a topological model for $\vd$. This will be used frequently to prove various results later in the article.

\begin{lem}\label{lem:impl'subset}
    Suppose $\mathcal{M}=\langle\mathcal{T},v\rangle$ be a topological model for $\vd$, where $\mathcal{T}=\langle X,\tau\rangle$ is a topological structure for $\vd$ and $v$ is a valuation. Then, for any $\alpha,\beta\in\lang$, $v(\alpha\lra\beta)=X$ iff $v(\alpha)\subseteq v(\beta)$.
\end{lem}
    
\begin{proof} 
    Suppose $v(\alpha\lra\beta)=X$. Then, $v(\alpha)^{c}\cup v(\beta)=X$. Let $a\in v(\alpha)$. So, $a\notin v(\alpha)^c$. Since $v(\alpha)^c\cup v(\beta)=X$, this implies that $a\in v(\beta)$. Thus, $v(\alpha)\subseteq v(\beta)$.

    Conversely, suppose $v(\alpha)\subseteq v(\beta)$ and let, if possible, $v(\alpha\lra\beta)\neq X$. Then, there exists $a\in X $ such that $a\notin v(\alpha\lra\beta)=v(\alpha)^{c}\cup v(\beta)$. So, $a\notin v(\alpha)^{c}$ and $a\notin v(\beta)$. Since $a\notin v(\alpha)^c$, $a\in v(\alpha)$. Then, as $v(\alpha) \subseteq v(\beta)$, $a\in v(\beta)$. This is a contradiction. Thus, $v(\alpha\lra\beta)=X$.
\end{proof}

\section{Soundness}\label{sec:soundness}
This section is devoted to showing that the logic $\vd$ is sound with respect to the topological semantics described in the previous section. We first prove the following topological property that will come in handy later in the proof of soundness.

\begin{lem}\label{lem:compclcomp}
    Suppose $\langle X,\tau\rangle$ is a topological space and $P\subseteq X$. Then, $\left(\ov{P^c}\right)^c=\Int(P)$.
\end{lem}

\begin{proof}
     Since $P^c\subseteq\ov{P^c}$, we have $\left(\ov{P^c}\right)^c\subseteq \left(P^c\right)^c=P$. Then, as $\left(\ov{P^c}\right)^c$ is an open set and $\Int(P)$ is the largest open set contained in $P$, we have $\left(\ov{P^c}\right)^c\subseteq\Int(P)$. 
     
     Now, suppose $x\in\Int(P)$. Then, there exists an open set $U\subseteq X$ such that $x\in U\subseteq P$. This implies that $P^c\subseteq U^c$, and so, $\ov{P^c}\subseteq\ov{U^c}$. Since $U^c$ is a closed set, $\ov{U^c}=U^c$. Hence, $\ov{P^c}\subseteq U^c$, which implies that $U\subseteq\left(\ov{P^c}\right)^c$. So, $x\in\left(\ov{P^c}\right)^c$, and hence, $\Int(P)\subseteq\left(\ov{P^c}\right)^c$. Thus, $\left(\ov{P^c}\right)^c=\Int(P)$.
\end{proof}

Next, we establish that the axioms of $\vd$ are valid and that the rules of inference preserve truth in the following two lemmas.

\begin{lem}\label{lem:ax-sound}
    Suppose $\varphi$ is an instance of an axiom of $\vd=\langle\lang,\vdash\rangle$ described in Section \ref{sec:logic}. Then, $\varphi$ is valid in $\vd$, i.e., $\modvd\varphi$. 
\end{lem}

\begin{proof}
Suppose $\mathcal{M}=\langle\mathcal{T},v\rangle$ be a topological model for $\vd$, where $\mathcal{T}=\langle X,\tau \rangle$ is a topological structure for $\vd$ and $v$ is a valuation. We need to show that $v(\varphi)=X$.
\begin{enumerate}
    \item Suppose $\varphi$ is an instance of Axiom \ref{ax:i}. Then, $\varphi$ is of the form $\alpha\lra(\beta\lra\alpha)$ for some $\alpha,\beta\in\lang$. Now,
    \[
    \begin{array}{ccl}
         v(\varphi)&=&v(\alpha)^c\cup v(\beta\lra\alpha)\\
         &=&v(\alpha)^c\cup (v(\beta)^c\cup v(\alpha))\\
         &=&X\cup v(\beta)^c\\
         &=&X
    \end{array}
    \]

     \item Suppose $\varphi$ is an instance of Axiom \ref{ax:ii}. Then, $\varphi$ is of the form $(\alpha\lra(\beta\lra\gamma))\lra ((\alpha\lra\beta)\lra(\alpha\lra\gamma))$ for some $\alpha,\beta,\gamma\in\lang$. Now, 
     \[
     v(\alpha\lra(\beta\lra\gamma))=v(\alpha)^c\cup v(\beta)^c\cup v(\gamma),
     \]
     and
     \[
     \begin{array}{ccl}
          v((\alpha\lra\beta)\lra(\alpha\lra\gamma))&=&v(\alpha\lra\beta)^c\cup v(\alpha\lra\gamma)\\
          &=&(v(\alpha)^c\cup v(\beta))^c\cup(v(\alpha)^c\cup v(\gamma))\\
          &=&(v(\alpha)\cap v(\beta)^c)\cup(v(\alpha)^c\cup v(\gamma))\\
          &=&(v(\alpha)\cup v(\alpha)^c)\cap(v(\beta)^c\cup v(\alpha)^c)\cup v(\gamma)\\
          &=&X\cap(v(\alpha)^c\cup v(\beta)^c\cup v(\gamma)\\
          &=&v(\alpha)^c\cup v(\beta)^c\cup v(\gamma)
     \end{array}
     \]
    Thus, $v(\alpha\lra(\beta\lra\gamma))=v((\alpha\lra\beta)\lra(\alpha\lra\gamma))$, and hence, by Lemma \ref{lem:impl'subset}, $v(\varphi)=v((\alpha\lra(\beta\lra\gamma))\lra ((\alpha\lra\beta)\lra(\alpha\lra\gamma)))=X$.     

     \item Suppose $\varphi$ is an instance of Axiom \ref{ax:iii}. Then, $\varphi$ is of the form $\alpha\lra(\beta\lra(\alpha\land\beta))$ for some $\alpha,\beta\in\lang$. Now, 
     \[
        \begin{array}{ccl}
             v(\beta\lra(\alpha\land\beta))&=&v(\beta)^c\cup v(\alpha\land\beta)\\
             &=&v(\beta)^c\cup(v(\alpha)\cap v(\beta))\\ 
             &=&(v(\beta)^c\cup v(\alpha))\cap(v(\beta)^c\cup v(\beta))\\
             &=&(v(\alpha)\cup v(\beta)^c)\cap X\\
             &=&v(\alpha)\cup v(\beta)^c
        \end{array}
     \]
     Thus, $v(\alpha)\subseteq v(\beta\lra(\alpha\land\beta))$. Hence, by Lemma \ref{lem:impl'subset}, $v(\varphi)=v(\alpha\lra(\beta\lra(\alpha\land\beta)))=X$.
  
     \item Suppose $\varphi$ is an instance of Axiom \ref{ax:iv}. Then, $\varphi$ is of the form $(\alpha\land\beta)\lra \alpha$ for some $\alpha,\beta\in\lang$. Now,
     \[
     v(\alpha\land\beta)=v(\alpha)\cap v(\beta)\subseteq v(\alpha).
     \]
     So, by Lemma \ref{lem:impl'subset}, $v(\varphi)=v((\alpha\land\beta)\lra \alpha)=X$.
     
     \item Suppose $\varphi$ is an instance of Axiom \ref{ax:v}. So, $\varphi$ is of the form $(\alpha\land\beta)\lra \beta$ for some $\alpha,\beta\in\lang$. Then, by similar arguments as for Axiom \ref{ax:iv} above, $v(\varphi)=X$.
     
     \item Suppose $\varphi$ is an instance of Axiom \ref{ax:vi}. Then, $\varphi$ is of the form $\alpha\lra(\alpha\lor\beta)$ for some $\alpha,\beta\in\lang$. Now, $v(\alpha\lor\beta)$ is such that $v(\alpha)\cup v(\beta)\subseteq v(\alpha\lor\beta)$. So, $v(\alpha)\subseteq v(\alpha\lor\beta)$, and hence, by Lemma \ref{lem:impl'subset}, $v(\varphi)=v(\alpha\lra(\alpha\lor\beta))=X$.
     
     \item Suppose $\varphi$ is an instance of Axiom \ref{ax:vii}. So, $\varphi$ is of the form $\beta\lra(\alpha\lor\beta)$ for some $\alpha,\beta\in\lang$. Then, by similar arguments as for Axiom \ref{ax:vi} above, $v(\varphi)=X$.

     \item Suppose $\varphi$ is an instance of Axiom \ref{ax:viii}. Then, $\varphi$ is of the form $(\alpha\lra\beta)\lor\alpha$ for some $\alpha,\beta\in\lang$. Now, 
     \[
     v(\alpha\lra\beta)\cup v(\alpha)=v(\alpha)^c\cup v(\beta)\cup v(\alpha)=X\cup v(\beta)=X.
     \]
     Since $v(\alpha\lra\beta)\cup v(\alpha)\subseteq v((\alpha\lra\beta)\lor\alpha)$, this implies that $X\subseteq v((\alpha\lra\beta)\lor\alpha)=v(\varphi)$. Hence, $v(\varphi)=X$.
     
     \item Suppose $\varphi$ is an instance of Axiom \ref{ax:ix}. Then, $\varphi$ is of the form $\alpha\lor\neg\alpha$ for some $\alpha\in\lang$. Now, $v(\alpha)\cup v(\neg\alpha)\subseteq v(\alpha\lor\neg\alpha)$, i.e., $v(\alpha)\cup\ov{v(\alpha)^c}\subseteq v(\alpha\lor\neg\alpha)$. Since $v(\alpha)^c\subseteq\ov{v(\alpha)^c}$, we have
     \[
     X=v(\alpha)\cup v(\alpha)^c\subseteq v(\alpha)\cup\ov{v(\alpha)^c}\subseteq v(\alpha\lor\neg\alpha).
     \]
     This implies that $v(\varphi)=v(\alpha\lor\neg\alpha)=X$.
     
     \item Suppose $\varphi$ is an instance of Axiom \ref{ax:x}. Then, $\varphi$ is of the form $(\alpha\lra\gamma)\lra((\neg\alpha\lra\gamma)\lra((\alpha\lor\neg\alpha)\lra\gamma))$ for some $\alpha,\gamma\in\lang$. From the validity of Axiom \ref{ax:ix} above, $v(\alpha\lor\neg\alpha)=X$. So, 
     \[
     v((\alpha\lor\neg\alpha)\lra\gamma))=v(\alpha\lor\neg\alpha)^c\cup v(\gamma)=X^c\cup v(\gamma)=v(\gamma).
     \]
     Now, $v(\neg\alpha\lra\gamma)=v(\neg\alpha)^c\cup v(\gamma)=\left(\ov{v(\alpha)^c}\right)^c\cup v(\gamma)$. Thus,
     \[
     \begin{array}{ccl}
         v((\neg\alpha\lra\gamma)\lra((\alpha\lor\neg\alpha)\lra\gamma))&=&(v(\neg\alpha\lra\gamma)^c\cup v((\alpha\lor\neg\alpha)\lra\gamma))\\
        &=&\left(\left(\ov{v(\alpha)^c}\right)^c\cup v(\gamma)\right)^c\cup v(\gamma)\\
        &=&\left(\ov{v(\alpha)^c}\cap v(\gamma)^c\right)\cup v(\gamma)\\
        &=&\left(\ov{v(\alpha)^c}\cup v(\gamma)\right)\cap X\\
        &=&\ov{v(\alpha)^c}\cup v(\gamma).
     \end{array}
     \]
     Now, as $v(\alpha)^c\subseteq\ov{v(\alpha)^c}$, we have
     \[
     v(\alpha\lra\gamma)=v(\alpha)^c\cup v(\gamma)\subseteq \ov{v(\alpha)^c}\cup v(\gamma)=v((\neg\alpha\lra\gamma)\lra((\alpha\lor\neg\alpha)\lra\gamma)).
     \]
     This implies, by Lemma \ref{lem:impl'subset}, that $v(\varphi)=X$.
     
     \item Suppose $\varphi$ is an instance of Axiom \ref{ax:xi}. Then, $\varphi$ is of the form $((\alpha\lra\beta)\lra\gamma)\lra((\alpha\lra\gamma)\lra(((\alpha\lra\beta)\lor\alpha)\lra\gamma))$ for some $\alpha,\beta,\gamma\in\lang$. Now, from the validity of Axiom \ref{ax:viii}, $v((\alpha\lra\beta)\lor\alpha)=X$. So, 
     \[
     v(((\alpha\lra\beta)\lor\alpha)\lra\gamma)=v((\alpha\lra\beta)\lor\alpha)^c\cup v(\gamma)=X^c\cup v(\gamma)=v(\gamma).
     \]
     Then,
     \[
     \begin{array}{ccl}
          v((\alpha\lra\gamma)\lra(((\alpha\lra\beta)\lor\alpha)\lra\gamma))&=&v(\alpha\lra\gamma)^c\cup v(((\alpha\lra\beta)\lor\alpha)\lra\gamma)\\
          &=&\left(v(\alpha)^c\cup v(\gamma)\right)^c\cup v(\gamma)\\
          &=&\left(v(\alpha)\cap v(\gamma)^c\right)\cup v(\gamma)\\
          &=&\left(v(\alpha)\cup v(\gamma)\right)\cap X\\
          &=&v(\alpha)\cup v(\gamma).
     \end{array}      
     \]
     Now, 
     \[
     v((\alpha\lra\beta)\lra\gamma)=v(\alpha\lra\beta)^c\cup v(\gamma)=\left(v(\alpha)^c\cup v(\beta)\right)^c\cup v(\gamma)=\left(v(\alpha)\cap v(\beta)^c\right)\cup v(\gamma).
     \]
     Since, $v(\alpha)\cap v(\beta)^c\subseteq v(\alpha)$,
     \[
     v((\alpha\lra\beta)\lra\gamma)\subseteq v((\alpha\lra\gamma)\lra(((\alpha\lra\beta)\lor\alpha)\lra\gamma)).
     \]
     This implies, by Lemma \ref{lem:impl'subset}, that $v(\varphi)=X$.

     \item Suppose $\varphi$ is an instance of Axiom \ref{ax:bc1}. So, $\varphi$ is of the form $\circ\alpha\lra(\alpha\lra(\neg\alpha\lra\beta))$ for some $\alpha,\beta\in\lang$. Now, $v(\circ\alpha)=v(\alpha)^c\cup\Int(v(\alpha))$ and
     \[
     \begin{array}{ccl}
          v(\alpha\lra(\neg\alpha\lra\beta))&=&v(\alpha)^c\cup v(\neg\alpha\lra\beta)\\
          &=&v(\alpha)^c\cup\left(v(\neg\alpha)^c\cup v(\beta)\right)\\
          &=&v(\alpha)^c\cup\left(\left(\ov{v(\alpha)^c}\right)^c\cup v(\beta)\right).
     \end{array}
     \]
     Then, by Lemma \ref{lem:compclcomp},
     \[
     v(\alpha\lra(\neg\alpha\lra\beta))=v(\alpha)^c\cup\Int(v(\alpha))\cup v(\beta).
     \]
     Thus, $v(\circ\alpha)\subseteq v(\alpha\lra(\neg\alpha\lra\beta))$, and hence, by Lemma \ref{lem:impl'subset}, $v(\varphi)=X$.

     \item Suppose $\varphi$ is an instance of Axiom \ref{ax:ciw}. So, $\varphi$ is of the form $\circ\alpha\lor(\alpha\land\neg\alpha)$ for some $\alpha\in\lang$. Now,
     \[
     \begin{array}{ccl}
          v(\circ\alpha)\cup v(\alpha\land\neg\alpha)&=&v(\alpha)^c\cup\Int(v(\alpha))\cup(v(\alpha)\cap v(\neg\alpha))\\
          &=&v(\alpha)^c\cup\Int(v(\alpha))\cup\left(v(\alpha)\cap\ov{v(\alpha)^c}\right)\\
          &=&\left(v(\alpha)^c\cup\Int(v(\alpha))\cup v(\alpha)\right)\cap\left(v(\alpha)^c\cup\Int(v(\alpha))\cup\ov{v(\alpha)^c}\right)\\
          &=&X\cap\left(v(\alpha)^c\cup\Int(v(\alpha))\cup\ov{v(\alpha)^c}\right)\\
          &=&v(\alpha)^c\cup\Int(v(\alpha))\cup\ov{v(\alpha)^c}\\
     \end{array}
     \]
     Since $\Int(v(\alpha))=\left(\ov{v(\alpha)^c}\right)^c$, by Lemma \ref{lem:compclcomp}, $v(\circ\alpha)\cup v(\alpha\land\neg\alpha)=X$. Thus,
     \[
     X=v(\circ\alpha)\cup v(\alpha\land\neg\alpha)\subseteq v(\circ\alpha\lor(\alpha\land\neg\alpha))=v(\varphi),
     \]
    which implies that $v(\varphi)=X$.
    
     \item Suppose $\varphi$ is an instance of Axiom \ref{ax:xiv}. So, $\varphi$ is of the form $(\circ\alpha\lra\gamma)\lra(((\alpha\land\neg\alpha)\lra\gamma)\lra((\circ\alpha\lor(\alpha\land\neg\alpha))\lra\gamma))$ for some $\alpha,\gamma\in\lang$. Now, by the validity of Axiom \ref{ax:ciw}, $v(\circ\alpha\lor(\alpha\land\neg\alpha))=X$. So, 
     \[
     v((\circ\alpha\lor(\alpha\land\neg\alpha))\lra\gamma)=v(\circ\alpha\lor(\alpha\land\neg\alpha))^c\cup v(\gamma)=X^c\cup v(\gamma)=v(\gamma).
     \]
     Now, 
     \[
     \begin{array}{ccl}
          v((\alpha\land\neg\alpha)\lra\gamma)&=&v(\alpha\land\neg\alpha)^c\cup v(\gamma)\\
          &=&(v(\alpha)\cap v(\neg\alpha))^c\cup v(\gamma)\\
          &=&\left(v(\alpha)\cap\ov{v(\alpha)^c}\right)^c\cup v(\gamma)\\
          &=&v(\alpha)^c\cup\left(\ov{v(\alpha)^c}\right)^c\cup v(\gamma)\\
     \end{array}
     \]
     Thus,
     \[
     \begin{array}{ccl}
          v(((\alpha\land\neg\alpha)\lra\gamma)\lra((\circ\alpha\lor(\alpha\land\neg\alpha))\lra\gamma))&=&\left(v(\alpha)^c\cup\left(\ov{v(\alpha)^c}\right)^c\cup v(\gamma)\right)^c\cup v(\gamma)\\
          &=&\left(v(\alpha)\cap\ov{v(\alpha)^c}\cap v(\gamma)^c\right)\cup v(\gamma)\\
          &=&\left(\left(v(\alpha)\cap\ov{v(\alpha)^c}\right)\cup v(\gamma)\right)\cap X\\
          &=&\left(v(\alpha)\cap\ov{v(\alpha)^c}\right)\cup v(\gamma).
     \end{array}
     \]
     Finally,
     \[
     \begin{array}{ccl}
          v(\circ\alpha\lra\gamma)&=&v(\circ\alpha)^c\cup v(\gamma)\\
          &=&\left(v(\alpha)^c\cup\Int(v(\alpha))\right)^c\cup v(\gamma)\\
          &=&\left(v(\alpha)\cap\left(\Int(v(\alpha))\right)^c\right)\cup v(\gamma)\\
          &=&\left(v(\alpha)\cap\ov{v(\alpha)^c}\right)\cup v(\gamma),\hbox{ by Lemma \ref{lem:compclcomp}.}
     \end{array}
     \]
     Thus, $v(\circ\alpha\lra\gamma)= v(((\alpha\land\neg\alpha)\lra\gamma)\lra((\circ\alpha\lor(\alpha\land\neg\alpha))\lra\gamma))$, which implies, by Lemma \ref{lem:impl'subset}, that $v(\varphi)=X$.

     \item Suppose $\varphi$ is an instance of Axiom \ref{ax:xv}. So, $\varphi$ is of the form $\sim\neg\alpha\lra\sim\neg\sim\neg\alpha$ for some $\alpha\in\lang$. Now,
     \[
     v(\sim\neg\alpha)=v(\neg\alpha)^c=\left(\ov{v(\alpha)^c}\right)^c=\Int(v(\alpha)),\hbox{ by Lemma \ref{lem:compclcomp}.}
     \]
     Thus, 
     \[
     v(\sim\neg\sim\neg\alpha)=\Int(v(\sim\neg\alpha))=\Int(\Int(v(\alpha)))=\Int(v(\alpha)).
     \]
     Since $v(\sim\neg\alpha)=v(\sim\neg\sim\neg\alpha)$, by Lemma \ref{lem:impl'subset}, $v(\varphi)=X$.
     
     \item Suppose $\varphi$ is an instance of Axiom \ref{ax:xvi}. So, $\varphi$ is of the form  $\sim\neg\sim\neg\alpha\lra\sim\neg\alpha$ for some $\alpha\in\lang$. By the arguments for the validity of Axiom \ref{ax:xv} above, $v(\sim\neg\sim\neg\alpha)=v(\sim\neg\alpha)$. So, by Lemma \ref{lem:impl'subset}, $v(\varphi)=X$.
     
     \item Suppose $\varphi$ is an instance of Axiom \ref{ax:xvii}. So, $\varphi$ is of the form $\sim\neg(\alpha\land\beta)\lra(\sim\neg\alpha\land\sim\neg\beta)$, for some $\alpha,\beta\in\lang$. Now, by the arguments for the validity of Axiom \ref{ax:xv}, 
     \[
     v(\sim\neg(\alpha\land\beta))=\Int(v(\alpha\land\beta))=\Int(v(\alpha)\cap v(\beta))=\Int(v(\alpha))\cap\Int(v(\beta))\footnote{It is easy to show that, for any topological space $\langle X,\tau\rangle$ and $P,Q\subseteq X$, $\Int(P\cap Q)=\Int(P)\cap\Int(Q)$.}.
     \]
     Also, 
     \[
     v(\sim\neg\alpha\land\sim\neg\beta)=v(\sim\neg\alpha)\cap v(\sim\neg\beta)=\Int(v(\alpha))\cap\Int(v(\beta)).
     \]
     Thus, $v(\sim\neg(\alpha\land\beta))=v(\sim\neg\alpha\land\sim\neg\beta)$, which implies, by Lemma \ref{lem:impl'subset}, that $v(\varphi)=X$.
     
     \item Suppose $\varphi$ is an instance of Axiom \ref{ax:xviii}. So, $\varphi$ is of the form $\sim\neg\alpha\land\sim\neg\beta\lra\sim\neg(\alpha\land\beta)$ for some $\alpha,\beta\in\lang$.  By the arguments for the validity of Axiom \ref{ax:xvii} above, $v(\sim\neg(\alpha\land\beta))=v(\sim\neg\alpha\land\sim\neg\beta)$. So, by Lemma \ref{lem:impl'subset}, $v(\varphi)=X$.
\end{enumerate}
Since $\mathcal{M}$ is an arbitrary topological model, this proves that every instance of an axiom of $\vd$ is valid.
\end{proof}

\begin{lem}\label{lem:rule-truthpreservation}
    Suppose $\mathcal{M}=\langle\mathcal{T},v\rangle$ be a topological model for $\vd$, where $\mathcal{T}=\langle X,\tau\rangle$ is a topological structure for $\vd$ and $v:\lang\to X$ is a valuation. For any $\alpha,\beta\in\lang$,
    \begin{enumerate}[label=(\roman*)]
        \item if $v(\alpha)=v(\alpha\lra\beta)=X$, then $v(\beta)=X$, i.e., Rule \ref{rule:mp} (MP) preserves truth.
        \item if $v(\alpha)=X$, then $v(\neg\alpha\lra\sim\alpha)=X$, i.e., Rule \ref{rule:neg} preserves truth.
    \end{enumerate}
\end{lem}

\begin{proof}
    \begin{enumerate}[label=(\roman*)]
        \item Suppose $v(\alpha)=v(\alpha\lra\beta)=X$. Then, 
        \[
        X=v(\alpha\lra\beta)=v(\alpha)^c\cup v(\beta)=X^c\cup v(\beta)=v(\beta).
        \]
       
        \item Suppose $v(\alpha)=X$. Then,
        \[
        v(\neg\alpha\lra\sim\alpha)=v(\neg\alpha)^c\cup v(\alpha)^c=\left(\ov{v(\alpha)^c}\right)^c\cup\emptyset=\left(\ov{\emptyset}\right)^c=X.
        \]
    \end{enumerate}
\end{proof}

\begin{thm}[\textsc{Soundness}]\label{thm:soundness}
For any $\Gamma\cup\{\alpha\}\subseteq \mathcal{L}$, if $\Gamma\vdvd\alpha$, then $\Gamma\modvd\alpha$.
\end{thm}

\begin{proof} 
Suppose $\Gamma\vdvd\alpha$. Let $\mathcal{M}=\langle\mathcal{T},v\rangle$ be a topological model for $\vd$, where $\mathcal{T}=\langle X,\tau\rangle$ is a topological structure for $\vd$ and $v$ is a valuation. 

Since $\Gamma\vdvd\alpha$, there exists a finite sequence $(\varphi_1,\ldots,\varphi_n=\alpha)$, such that, for each $1\le i\le n$,
\begin{enumerate}[label=(\roman*)]
    \item $\varphi_{i}$ is an instance of an axiom of $\vd$, or
    \item $\varphi_{i}\in \Gamma$, or
    \item there exist $1\le j,k<i$ such that $\varphi_{i}$ is obtained from $\varphi_{j},\varphi_{k}$ by MP.
    \item there exists  $j<i$ such that $\varphi_{i}$ is the result of an application of the second inference rule on $\varphi_j$.
\end{enumerate}
Let $\Gamma_0=\{\varphi_i\mid\,1\le i\le n\hbox{ and } \varphi_i\in\Gamma\}$. Then, $\Gamma_0$ is a finite subset of $\Gamma$.
         
\textsc{Case 1:} $\Gamma_0=\emptyset$.

Then, $\vdvd\alpha$ and, for each $1\le i\le n$,
\begin{enumerate}[label=(\roman*)]
    \item $\varphi_{i}$ is an instance of an axiom of $\vd$, or
    \item there exist $1\le j,k<i$ such that $\varphi_{i}$ is obtained from $\varphi_{j},\varphi_{k}$ by MP.
    \item there exists  $j<i$ such that $\varphi_{i}$ is the result of an application of the second inference rule on $\varphi_j$.
\end{enumerate}
Thus, by Lemmas \ref{lem:ax-sound} and \ref{lem:rule-truthpreservation} and by a straightforward induction, $v(\varphi_i)=X$ for all $1\le i\le n$. Hence, $v(\varphi_n)=v(\alpha)=X$.

This, in fact, establishes that every theorem of $\vd$ is valid.

\textsc{Case 2:} $\Gamma_0\neq\emptyset$.         

Let $\Gamma_{0}=\{\gamma_{1},\gamma_2,\ldots,\gamma_{k}\}$. Then, $\Gamma_0\vdvd\alpha$, i.e., $\{\gamma_{1},\gamma_{2},\ldots,\gamma_{k}\}\vdvd\alpha$. So, by applying the Deduction theorem (Theorem \ref{thm:DedThm}) repeatedly, we get 
\[
\vdvd\gamma_1\lra(\gamma_{2}\lra\cdots(\gamma_{k}\lra\alpha)\cdots).
\]
Thus, by Case 1, $\modvd\gamma_1\lra(\gamma_{2}\lra\cdots(\gamma_{k}\lra\alpha)\cdots)$. This implies that $v(\gamma_{1}\lra(\gamma_{2}\lra\cdots(\gamma_{k}\lra\alpha)\cdots))=X$. Now,
\[
v(\gamma_{1}\lra(\gamma_{2}\lra\cdots(\gamma_{k}\lra\alpha)\cdots))=v(\gamma_{1})^{c}\cup v(\gamma_{2}\lra\cdots(\gamma_{k}\lra\alpha)\cdots)).
\]
Then, proceeding similarly, we get
\[
\begin{array}{ccl}
     v(\gamma_{1}\lra(\gamma_{2}\lra\cdots(\gamma_{k}\lra\alpha)\cdots))&=&\left(v(\gamma_{1})^{c}\cup v(\gamma_{2})^{c}\cup\dots v(\gamma_{k})^{c}\right)\cup v(\alpha)\\
     &=&\left(v(\gamma_1)\cap v(\gamma_{2})\cap\dots\cap v(\gamma_{k})\right)^{c}\cup v(\alpha)\\
     &=&v((\gamma_1\land\gamma_2\land\cdots\land\gamma_k)\lra\alpha)\\
\end{array}
\] 
Thus, $v((\gamma_1\land\gamma_2\land\cdots\land\gamma_k)\lra\alpha)=X$, which implies, by Lemma \ref{lem:impl'subset}, that $v(\gamma_1\land\gamma_2\land\cdots\gamma_k)\subseteq v(\alpha)$, i.e., $\displaystyle\bigcap_{\gamma\in\Gamma_0}v(\gamma)\subseteq v(\alpha)$. 

Since $\mathcal{M}$ is an arbitrary topological model, this implies, by Definition \ref{def:sem'tic truth}, that $\Gamma\modvd\alpha$.
\end{proof}  

\section{\texorpdfstring{$\vd$}{\textbf{vD}}: a non-self-extensional LFI}\label{sec:lfi}

The \emph{Logics of Formal inconsistency} or \emph{LFIs} are paraconsistent logics, i.e., logics with a non-explosive negation. Along with that, these logics use a unary \emph{consistency} operator, usually denoted by $\circ$. The following is a simplified definition of an LFI. A more general definition can be found in \cite{Carnielli2007,CarnielliConiglio2016}.

\begin{dfn}\label{def:lfi}
Let $\mathscr{L}=\langle\lang,\vdash\rangle$ be a standard logic with a signature containing a negation $\neg$ and a primitive or defined unary consistency operator $\circ$. Then, $\mathscr{L}$ is said to be a \emph{Logic of Formal Inconsistency (LFI)} with respect to $\neg$ and $\circ$ if the following conditions hold.

\begin{enumerate}[label=(\roman*)]
    \item $\{\varphi,\neg\varphi\}\not\vdash\psi$ for some $\varphi,\psi\in\lang$.
    \item There exist $\varphi,\psi\in\lang$ such that \begin{enumerate}
        \item $\{\circ\varphi,\varphi\}\not\vdash\psi$;
        \item $\{\circ\varphi,\neg\varphi\}\not\vdash\psi$.
    \end{enumerate}
    \item $\{\circ\varphi,\varphi,\neg\varphi\}\vdash\psi$ for all $\varphi,\psi\in\lang$.
\end{enumerate}
\end{dfn}

\begin{thm}\label{thm:vd_lfi}
    Suppose $\vd=\langle\lang,\vdvd\rangle$ as before, and $V$ be the set of variables. Let $p,q\in V$ be two distinct variables. Then, the following statements hold.
    \begin{enumerate}[label=(\roman*)]
        \item $\{p,\neg p\}\not\vdvd q$;
        \item $\{\circ p,p\}\not\vdvd q$;
        \item $\{\circ p,\neg p\}\not\vdvd q$.
    \end{enumerate}
\end{thm}
    
\begin{proof}  
    Let $\mathcal{M}=\langle\langle\mathbb{R},\tau\rangle,v\rangle$, be a topological model for $\vd$, where $\langle\mathbb{R},\tau\rangle$ is the topological space on $\mathbb{R}$, the set of real numbers and $\tau$ is the usual topology on $\mathbb{R}$. Moreover, let $v$ be a valuation for $\vd$ such that $v(p)=[0,1)$ and $v(q)=(2,3)$.
    \begin{enumerate}[label=(\roman*)]
        \item We note that $v(p)^{c}=(-\infty,0)\cup[1,\infty)$, which implies that $v(\neg p)=\ov{v(p)^c}=(-\infty,0]\cup[1,\infty)$. So,
        \[
        \begin{array}{rll}
             v(p\lra(\neg p\lra q))&=&v(p)^c\cup v(\neg p)^c\cup v(q)\\
             &=&(-\infty,0)\cup[1,\infty)\cup(0,1)\cup(2,3)\\
             &=&(-\infty,0)\cup(0,\infty)\\
             &\neq&\mathbb{R}
        \end{array}
        \]
       Thus, $\not\modvd p\lra(\neg p\lra q)$, and hence, by the Soundness theorem (Theorem \ref{thm:soundness}), $\not\vdvd  p\lra(\neg p\lra q)$. Hence, by the Deduction theorem (Theorem \ref{thm:DedThm}), $\{p,\neg p\}\not\vdvd q$.
       
        \item We note that $v(\circ p)=v(p)^c\cup \mathrm{Int}(p)=\mathbb{R}\mbox{\textbackslash}\{0\}$. Then,
        \[
        \begin{array}{rcl}
             v(\circ p\lra(p\lra q))&=&v(\circ p)^c\cup v(p)^c\cup v(q)\\
             &=&\{0\}\cup(-\infty,0)\cup[1,\infty)\cup(2,3)\\
             &=&(-\infty,0]\cup[1,\infty)\\
             &\neq&\mathbb{R}
        \end{array}
        \]
        Thus, $\not\modvd\circ p\lra(p\lra q)$, and hence, by the Soundness theorem (Theorem \ref{thm:soundness}), $\not\vdvd  \circ p\lra(p\lra q)$. Hence, by the Deduction theorem (Theorem \ref{thm:DedThm}), $\{\circ p,p\}\not\vdvd q$.
        
        \item We note that
        \[
        \begin{array}{rcl}
             v(\circ p\lra(\neg p\lra q))&=&v(\circ p)^c\cup v(\neg p)^c\cup v(q)\\
             &=&\{0\}\cup(0,1)\cup(2,3)\\
             &=&[0,1)\cup(2,3)\\
             &\neq&\mathbb{R}
        \end{array}
        \]
        Thus, $\not\modvd\circ p\lra(\neg p\lra q)$, and hence, by the Soundness theorem (Theorem \ref{thm:soundness}), $\not\vdvd  \circ p\lra(\neg p\lra q)$. Hence, by the Deduction theorem (Theorem \ref{thm:DedThm}), $\{\circ p,\neg p\}\not\vdvd q$.
    \end{enumerate}
\end{proof}

\begin{cor}
    $\vd$ is an LFI.
\end{cor}

\begin{proof}
    It follows from Theorem \ref{thm:vd_lfi}, that the conditions (i) and (ii) of Definition \ref{def:lfi} are satisfied by $\vd$. Now, by Axiom \ref{ax:bc1}, $\vdvd\circ \varphi\lra(\varphi\lra(\neg\varphi\lra \psi))$ for any $\varphi\in\lang$. So, by the Deduction theorem (Theorem \ref{thm:DedThm}),  $\{\varphi,\neg\varphi,\circ\varphi\}\vdvd\psi$ for any $\varphi,\psi\in\lang$. Thus, the condition (iii) of Definition \ref{def:lfi} is also satisfied by $\vd$. Hence, $\vd$ is an LFI.
\end{proof}

We now introduce an abbreviation as follows. Suppose $\logl=\langle\lang,\vdash\rangle$. The relation $\equiv$ is defined on $\lang$ as follows. For any $\alpha,\beta\in\lang$,
\[
\alpha\equiv\beta \qquad\hbox{ iff }\qquad \{\alpha\}\vdash\beta\;\hbox{ and }\;\{\beta\}\vdash\alpha.
\]

\begin{dfn}\label{def:rep'ment} Suppose $\logl=\langle\lang,\vdash\rangle$ is a logic with $V$ as the set of variables over which $\lang$ is generated. Then,the following property is referred to as \emph{replacement}\footnote{This property is referred to as replacement in \cite{Carnielli2007} and weak replacement in \cite{CarnielliConiglio2016, ConiglioPrieto-Sanabria2017}.}. For $\alpha_i,\beta_i\in\lang$, $i=1,\ldots,n$, if $\alpha_i\equiv\beta_i$, then for any $\varphi(p_1,\cdots,p_n)\in\lang$ involving variables $p_1,\ldots,p_n\in V$,
\[
\varphi(p_{1}/\alpha_{1},\dots, p_{n}/\alpha_{n})\equiv\varphi(p_{1}/\beta_{1},\dots,p_{n}/\beta_{n}),
\]
where $\varphi(p_{1}/\psi_{1},\dots, p_{n}/\psi_{n})$ denotes the formula obtained by replacing each occurrence of $p_i$  by $\psi_i$ in $\varphi$.

The logic $\logl$ is said to be \emph{self-extensional} (according to the terminology introduced in \cite{Wojcicki1988}), if the above replacement property holds in it, and \emph{non-self-extensional} otherwise.
\end{dfn}

\begin{rem}
    $\vd=\langle\lang,\vdvd\rangle$ is non-self-extensional, i.e., the replacement property fails in $\vd$. To see this, let $\alpha_1,\alpha_2,\beta_1,\beta_2\in\lang$ such that $\alpha_1\equiv\alpha_2$ and $\beta_1\equiv\beta_2$. Let $\mathcal{M}=\langle\langle X,\tau\rangle,v\rangle$, where $\langle X,\tau\rangle$ is a topological structure for $\vd$ and $v$ is a valuation, be a topological model for $\vd$ such that $v(\alpha_1),v(\alpha_2),v(\beta_1),v(\beta_2)\neq X$. Then, by the Soundness theorem (Theorem \ref{thm:soundness}), $v(\alpha_1)=v(\alpha_2)$ and $v(\beta_1)=v(\beta_2)$. Now, $v(\alpha_1\lor\beta_1)$ and $v(\alpha_2\lor\beta_2)$ are such that $v(\alpha_1\lor\beta_2)\supseteq v(\alpha_1)\cup v(\beta_1)$ and $v(\alpha_2\lor\beta_2)\supseteq v(\alpha_2)\cup v(\beta_2)$. Although, $v(\alpha_1)\cup v(\beta_1)=v(\alpha_2)\cup v(\beta_2)$, this does not necessarily imply that $v(\alpha_1\lor\beta_1)=v(\alpha_2\lor\beta_2)$. Hence, $\vd$ is non-self-extensional.
\end{rem}

\section{Completeness}\label{sec:com'ness}
In this section, we prove that the logic $\vd$ is complete with respect to the topological semantics described earlier. Before we dive into the theorem itself, we mention a few definitions and results that will be useful for the proof of the Completeness theorem.

\subsection{Kuratowski operators}
\begin{dfn}\label{def:Kuratowski}
    Suppose $X$ is a set. Then, an operator $\ov{(\cdot)}:\pow(X)\to\pow(X)$ is called a \emph{Kuratowski closure operator over $X$} if the following conditions are satisfied.
    \begin{enumerate}[label=(\arabic*)]
        \item $\ov{\emptyset}=\emptyset$;
        \item $A\subseteq \ov{A}$ for all $A\subseteq X$;
        \item $\ov{A\cup B}=\ov{A}\cup\ov{B}$ for all $A,B\subseteq X$;
        \item $\ov{\ov{A}}=\ov{A}$, for all $A\subseteq X$.
    \end{enumerate}
\end{dfn}

It is well-known that a Kuratowski closure operator over $X$ defines a unique topology on $X$. One can also define a Kuratowski-like operator on certain collections of subsets of $X$ and extend this operator to a Kuratowski closure operator over $X$. This is described below.

\begin{dfn}\label{def:kuratowski-like}
    Suppose $X$ is a set and $\mathcal{B}\subseteq\pow(X)$ such that
\begin{enumerate}
    \item $\emptyset,X\in \mathcal{B}$, and
    \item for any $F,G\in \mathcal{B}$, $F\cup G\in \mathcal{B}$.    
\end{enumerate}
A \emph{Kuratowski-like operator on $\mathcal{B}$} is a map $\wh{(\cdot)}:\mathcal{B}\lra\mathcal{B}$ such that the following conditions are satisfied.
\begin{enumerate}
    \item $\wh{\emptyset}=\emptyset$;
    \item $F\subseteq\wh{F}$ for every $F\in \mathcal{B}$;
    \item $\wh{F\cup G}=\wh{F}\cup\wh{G}$ for every $F,G\in \mathcal{B}$;
    \item $\wh{\wh{F}}=\wh{F}$ for every $F\in \mathcal{B}$.
\end{enumerate}
\end{dfn}

The following theorem says that a Kuratowski-like operator on a set $\mathcal{B}\subseteq\pow(X)$ can be extended to a Kuratowski closure operator over $X$. A proof of the theorem can be found in \cite{ConiglioPrieto-Sanabria2017}.

\begin{thm}\label{thrm:kura'ki result}
    Suppose $X$ is a set and $\mathcal{B}\subseteq\pow(X)$ is as described above. Let $\wh{(\cdot)}:\mathcal{B}\to\mathcal{B}$ be a Kuratowski-like operator on $\mathcal{B}$. Then, the map $\ov{(\cdot)}:\pow(X)\lra\pow(X)$, defined by $\overline{A}=\bigcap\{\wh{F}\mid\,F\in\mathcal{B}\mbox{ and }A\subseteq\wh{F}\}$, for any $A\subseteq X$, is a Kuratowski closure operator over $X$. Moreover, $\overline{F}=\wh{F}$ for all $F\in \mathcal{B}$.    
\end{thm} 

\subsection{\texorpdfstring{$\alpha$}{a}-saturated sets}

\begin{dfn}\label{def:alp sa'ted}
    Suppose $\mathscr{L}=\langle\lang,\vdash\rangle$ is a logic and $\Delta\cup\{\alpha\}\subseteq\lang$. $\Delta$ is said to be  \emph{$\alpha$-saturated} in $\mathscr{L}$ if the following conditions hold.
    \begin{enumerate}[label=(\alph*)]
        \item $\Delta\not\vdash\alpha$;
        \item $\Delta\cup\{\beta\}\vdash\alpha$ for any $\beta\notin\Delta$.
    \end{enumerate}
\end{dfn}

\begin{dfn}
    Suppose $\mathscr{L}=\langle\lang,\vdash\rangle$ is a logic. A set $\Gamma\subseteq\lang$ is said to be \emph{closed} if, for any $\alpha\in\lang$, $\Gamma\vdash\alpha$ iff $\alpha\in\Gamma$.
\end{dfn}

\begin{lem}\label{lem:maxrel->closed}
    Suppose $\mathscr{L}=\langle\mathcal{L},\vdash\rangle$ is a Tarskian logic and $\Gamma\subseteq \mathcal{L}$. If $\Gamma$ is an  $\alpha$-saturated set, for some $\alpha\in\lang$, then $\Gamma$ is closed.
\end{lem}

\begin{proof}
    Suppose $\Gamma$ is $\alpha$-saturated, where $\alpha\in\lang$, but not closed. Now, as $\lang$ is Tarskian, and hence, satisfies reflexivity, $\Gamma\vdash\psi$ for all $\psi\in\Gamma$. Then, as $\Gamma$ is not closed, there exists $\psi\in\mathcal{L}$ such that $\Gamma\vdash\psi$ but $\psi\notin\Gamma$. Since $\Gamma$ is $\alpha$-saturated, this means $\Gamma\not\vdash\alpha$ but $\Gamma\cup\{\psi\}\vdash\alpha$. By reflexivity again, and the fact that $\Gamma\vdash\psi$, we have $\Gamma\vdash\theta$ for all $\theta\in\Gamma\cup\{\psi\}$. Then, by transitivity, and the fact that $\Gamma\cup\{\psi\}\vdash\alpha$, $\Gamma\vdash\alpha$. This is, however, a contradiction. Hence, $\Gamma$ must be closed.
\end{proof}

The following result is the well-known Lindenbaum Asser theorem. We include this without proof here. See \cite{Beziau1999}, or the more recent article \cite{BasuJain2025}, for a proof.

\begin{thm}\label{thm:Linde'os}
    Suppose $\mathscr{L}=\langle\lang,\vdash\rangle$ be a Tarskian and finitary logic. Let $\Gamma\cup\{\alpha\}\subseteq\lang$ such that $\Gamma\not\vdash\alpha$. Then, there exists $\Delta\supseteq\Gamma$ that is $\alpha$-saturated in $\mathscr{L}$.
\end{thm}

\begin{rem}\label{rem:vd Lin'os}
    Since the logic $\vd=\langle\lang,\vdvd\rangle$ is Tarskian and finitary, Lemma \ref{lem:maxrel->closed} and Theorem \ref{thm:Linde'os} hold in it.
\end{rem}

\begin{lem}\label{lem:satu'ed sets prop}
  Suppose $\Delta$ is an $\alpha$-saturated set in $\vd$, where $\alpha\in\lang$. Then, for any $\beta,\gamma\in\lang$, the following statements hold.
  \begin{enumerate}[label=(\roman*)]
      \item\label{prop:cl_neg} $\sim \beta\in \Delta$ iff, $\beta\notin \Delta$.
      \item\label{prop:arrow} $\beta\lra\gamma\in \Delta$ iff, either $\beta \notin \Delta$ or $\gamma\in \Delta$.
      \item\label{prop:and} $\beta \land \gamma \in \Delta$ iff, $\beta \in \Delta$ and $\gamma \in \Delta$.
      \item \label{prop:or} If either $\beta \in \Delta$ or $\gamma \in \Delta$, then $\beta \lor \gamma \in \Delta$.
      \item \label{prop:pc neg} If $\beta \notin \Delta$, then $\neg \beta \in \Delta$.
      \item \label{prop:ax in sat'set} $\beta \lor \neg \beta \in \Delta$ and $\sim \neg (\beta \lor \neg \beta) \in \Delta$.
      \item\label{prop:circ} $\circ\beta \in \Delta$ iff, $\beta\notin \Delta$ or $\neg\beta\notin \Delta$.
    \end{enumerate}
\end{lem}

\begin{proof}\begin{enumerate}[label=(\roman*)]
      \item  Suppose $\sim\beta\in\Delta$. So, by reflexivity, $\Delta\vdvd \sim\beta$. If possible, let $\beta\in\Delta$. Then, we can construct the following derivation of $\alpha $ from $\Delta$.
      \[
        \begin{array}{rll}
            1.&\beta&[\beta\in \Delta]\\
            2.&(\beta\lra(\sim \beta\lra\alpha))& [\hbox{Theorem \ref{thm:prop_cl_neg}\ref{ax:expl}}]\\
            3.&\sim\beta\lra\alpha&[\hbox{MP on (1) and (2)}]\\
            4.&\alpha&[\hbox{MP on (2) and (3)}]
        \end{array}
    \]
    Thus, $\Delta\vdvd\alpha$, which contradicts our assumption that $\Delta$ is an $\alpha$-saturated set. Hence, $\beta\notin\Delta$.

    Conversely, suppose $\beta\notin \Delta$. If possible, let $\sim \beta\notin \Delta$. Since $\Delta$ is $\alpha$-saturated and $\beta\notin\Delta$, $\Delta\cup\{\beta\}\vdvd\alpha$. Then, by the Deduction theorem (Theorem \ref{thm:DedThm}), $\Delta\vdvd\beta\lra\alpha$.

    By the same arguments, since $\sim\beta\notin\Delta$, $\Delta\vdvd\sim\beta\lra\alpha$.

    Now, as $\Delta$ is $\alpha$-saturated, it is closed, by Lemma \ref{lem:maxrel->closed}. So, $\beta\lra\alpha,\sim\beta\lra\alpha\in\Delta$.
 
    Then, we can construct the following derivation of $\alpha$ from $\Delta$.
    \[
        \begin{array}{rll}
            1.&(\sim\beta\lor\beta)& [\hbox{Theorem \ref{thm:prop_cl_neg}\ref{ax:lem}}]\\
            2.&((\beta\lra\bot)\lra\alpha)\lra((\beta\lra\alpha)\lra(((\beta\lra\bot)\lor\beta)\lra\alpha))&[\hbox{Axiom \ref{ax:xi}}]\\
            3.&(\sim\beta\lra\alpha)\lra((\beta\lra\alpha)\lra((\sim\beta\lor\beta)\lra\alpha))&[\hbox{Definition of $\sim$}]\\
            4.&\sim\beta\lra\alpha&[\sim\beta\lra\alpha\in\Delta]\\
            5.&(\beta\lra\alpha)\lra((\sim\beta\lor\beta)\lra\alpha)&[\hbox{MP on (3) and (4)}]\\
            6.&\beta\lra\alpha&[\beta\lra\alpha\in\Delta]\\
            7.&((\sim\beta\lor\beta)\lra\alpha)&[\hbox{MP on (5) and (6)}]\\
            8.&\alpha&[\hbox{MP on (1) and (7)}]
        \end{array}
    \]
So, $\Delta\vdvd\alpha$, which contradicts our assumption that $\Delta$ is an $\alpha$-saturated set. Thus, $\sim\beta\in\Delta$. Hence, $\sim\beta\in\Delta$ iff, $\beta\notin\Delta$.

\item Suppose $\beta\lra\gamma,\beta\in \Delta$. Then we can construct the following derivation of $\gamma$ from $\Delta$.
\[
    \begin{array}{rll}

    1.&\beta\lra\gamma&[\beta\lra\gamma\in \Delta]\\
    2.&\beta&[\beta\in\Delta]\\
    3.&\gamma&[\hbox{MP on (1) and (2)}]\\
    \end{array}
\]
Thus, $\Delta\vdvd\gamma$. Now, since $\Delta$ is $\alpha$-saturated, it is closed by Lemma \ref{lem:maxrel->closed}. So, $\gamma\in \Delta$. Hence, either $\beta\notin\Delta$ or $\gamma\in \Delta$.

Conversely, suppose either $\beta\notin \Delta$ or $\gamma\in \Delta$.

\textsc{Case 1:} Suppose $\beta\notin\Delta$. If possible, let $\beta\lra\gamma\notin\Delta$. Then, as $\Delta$ is $\alpha$-saturated, $\Delta\cup\{\beta\}\vdvd\alpha$  and $\Delta\cup\{\beta\lra\gamma\}\vdvd\alpha$. So, by the Deduction theorem (Theorem \ref{thm:DedThm}), $\Delta\vdvd\beta\lra\alpha$ and  $\Delta\vdvd(\beta\lra\gamma)\lra\alpha$. By Lemma \ref{lem:maxrel->closed}, since $\Delta$ is $\alpha$-saturated, it is closed. So, $\beta\lra\alpha, (\beta\lra\gamma)\lra\alpha\in\Delta$. We can now construct the following derivation of $\alpha$ from $\Delta$.
\[
    \begin{array}{rll}
    1.&((\beta\lra\gamma)\lra\alpha)\lra((\beta\lra\alpha)\lra(((\beta\lra\gamma)\lor\beta)\lra\alpha))&[\hbox{Axiom \ref{ax:xi}}]\\
    2.&(\beta\lra\gamma)\lra\alpha&[\hbox{From }\Delta]\\
    3.&(\beta\lra\alpha)\lra(((\beta\lra\gamma)\lor\beta)\lra\alpha)&[\hbox{MP on (1) and (2)}]\\
    4.&\beta\lra\alpha&[\hbox{From }\Delta]\\
    5.&((\beta\lra\gamma)\lor\beta)\lra\alpha&[\hbox{MP on (3) and (4)}]\\
    6.&(\beta\lra\gamma)\lor\beta&[\hbox{Axiom \ref{ax:viii}}]\\
    7. &\alpha&[\hbox{MP on (5) and (6)}]
    \end{array}
\]
So, $\Delta\vdvd\alpha$, which is not possible since $\Delta$ is an $\alpha$-saturated set. Thus, $\beta\lra\gamma\in\Delta$.

\textsc{Case 2:} Suppose $\gamma\in \Delta$. Then, we can construct the following derivation of $\beta\lra\gamma$ from $\Delta$.
\[
    \begin{array}{rll}
    1.&\gamma\lra(\beta\lra\gamma)&[\hbox{Axiom \ref{ax:i}}]\\
    2.&\gamma&[\gamma\in \Delta]\\
    3.&\beta\lra\gamma &[\hbox{MP on (1) and (2)}]
\end{array}
\]
So, $\Delta\vdvd\beta\lra\gamma$. Then, as $\Delta$ is $\alpha$-saturated, and hence, by Lemma \ref{lem:maxrel->closed}, closed, $\beta\lra\gamma\in\Delta$. Hence, $\beta\lra\gamma\in \Delta$ iff, either $\beta\notin\Delta$ or $\gamma\in\Delta$.

\item Suppose $\beta\land\gamma\in \Delta$. Then, we can construct the following derivation of $\beta$ from $\Delta$. 
\[
    \begin{array}{rll}
    1.&\beta\land\gamma&[\beta\land\gamma\in \Delta]\\
    2.&\beta\land\gamma\lra\beta&[\hbox{Axiom \ref{ax:iv}}]\\
    3.&\beta&[\hbox{MP on (1) and (2)}]
    \end{array}
\]
So, $\Delta\vdvd\beta$. Since $\Delta$ is an $\alpha$-saturated set, and hence, by Lemma \ref{lem:maxrel->closed}, closed, this implies that $\beta\in\Delta$. By similar arguments, $\gamma\in\Delta$.

Conversely, suppose $\beta,\gamma\in \Delta$. Then, we can construct the following derivation of $\beta\land\gamma$ from $\Delta$.
\[
    \begin{array}{rll}
    1.&\beta\lra(\gamma\lra(\beta\land\gamma))&[\hbox{Axiom \ref{ax:iii}}]\\
    2.&\beta&[\beta\in \Delta]\\
    3.&(\gamma\lra(\beta\land\gamma))&[\hbox{MP on (1) and (2)}]\\
    4.&\gamma&[\gamma\in\Delta]\\
    5.&\beta\land \gamma&[\hbox{MP on (3) and (4)}]
    \end{array}
\]
So, $\Delta\vdvd\beta\land\gamma$. Since $\Delta$ is an $\alpha$-saturated set, and hence, by Lemma \ref{lem:maxrel->closed}, closed, this implies that $\beta\land\gamma\in\Delta$.

\item Suppose either $\beta \in \Delta$ or $\gamma \in \Delta$.

If $\beta\in \Delta$, then we can construct the following derivation of $\beta\lor\gamma$ from $\Delta$. 
\[
    \begin{array}{rll}
    1.&\beta\lra(\beta\lor\gamma)&[\hbox{Axiom \ref{ax:vi}}]\\
    2.&\beta&[\beta\in\Delta]\\
    3.&\beta\lor\gamma&[\hbox{MP on (1) and (2)}]
    \end{array}
\]
Thus, $\Delta\vdvd\beta\lor\gamma$. By similar arguments, if $\gamma\in \Delta$, then $\Delta\vdvd\beta\lor\gamma$. Since $\Delta$ is an $\alpha$-saturated set, and hence, by Lemma \ref{lem:maxrel->closed}, closed, this implies that $\beta\lor\gamma\in\Delta$.

\item Suppose $\beta\notin \Delta$. If possible, let $\neg\beta\notin \Delta$. Then, $\Delta\cup\{\beta\}\vdvd\alpha$ and $\Delta\cup\{\neg\beta\}\vdvd\alpha$, since $\Delta$ is $\alpha$-saturated. So, by the Deduction theorem (Theorem \ref{thm:DedThm}), $\Delta\vdvd\beta\lra\alpha$ and $\Delta\vdvd\neg\beta\lra\alpha$. Since $\Delta$ is an $\alpha$-saturated set, and hence, by Lemma \ref{lem:maxrel->closed}, closed, this implies that $\beta \lra\alpha, \neg\beta\lra\alpha\in \Delta$. Then, we can construct the following derivation of $\alpha$ from $\Delta$. 
\[
    \begin{array}{rll}
    1.&(\beta\lra\alpha)\lra((\neg\beta\lra\alpha)\lra((\beta\lor\neg\beta)\lra\alpha))& [\hbox{Axiom \ref{ax:x}}]\\
    2.&\beta\lra\alpha&[\beta\lra\alpha\in \Delta]\\
    3.&(\neg\beta\lra\alpha)\lra((\beta\lor\neg\beta)\lra\alpha)&[\hbox{MP on (1) and (2)}]\\
    4.&\neg\beta\lra\alpha&[\neg\beta\lra\alpha\in\Delta]\\
    5.&((\beta\lor\neg\beta)\lra\alpha)&[\hbox{MP on (3) and (4)}]\\
    6.&\beta\lor\neg\beta&[\hbox{Axiom \ref{ax:ix}}]\\
    7.&\alpha&[\hbox{MP on (5) and (6)}]
    \end{array}
\]
So, $\Delta\vdvd \alpha$. This, however, contradicts our assumption that $\Delta$ is $\alpha$-saturated. Thus, if $\beta\notin \Delta$, then $\neg\beta\in \Delta$.

\item Since $\beta\lor \neg\beta$ is an axiom of $\vd$ (Axiom \ref{ax:ix}), $\vdvd\beta\lor\neg\beta$. So, by monotonicity, $\Delta\vdvd\beta\lor \neg\beta$. Since $\Delta$ is an $\alpha$-saturated set, and hence, by Lemma \ref{lem:maxrel->closed}, closed, this implies that $\beta\lor\neg\beta\in\Delta$.

Suppose $\sim\neg(\beta\lor \neg\beta)\notin \Delta$. So, by part \ref{prop:cl_neg}, $\neg(\beta\lor \neg\beta)\in \Delta$. Then, we have the following derivation of $\sim(\beta\lor\neg\beta)$ from $\Delta$.
\[
    \begin{array}{rll}
    1.&\neg(\beta\lor \neg\beta)&(\hbox{From } \Delta]\\ 
    2.&\neg(\beta\lor \neg\beta)\lra\,\sim(\beta\lor \neg\beta)&[\hbox{Rule \ref{rule:neg}, since $\beta\lor\neg\beta$ is a theorem}]\\
    3.&\sim(\beta\lor \neg\beta)&[\hbox{MP on (1) and (2)}]
    \end{array}
\]
So, $\Delta\vdvd\sim(\beta\lor\neg\beta)$. Again, since $\Delta$ is closed, this implies that $\sim(\beta\lor \neg\beta)\in \Delta$. Then, by part \ref{prop:cl_neg}, $\beta\lor\neg\beta\notin\Delta$. This, however, contradicts our earlier conclusion. Hence, $\sim\neg(\beta\lor \neg\beta)\in \Delta$.

\item Suppose $\circ\beta\in \Delta$. If possible, let $\beta,\neg\beta \in \Delta$. Then, we can construct the following derivation of  $\alpha $ from $\Delta$.
\[
  \begin{array}{rll}
  1.&\circ\beta\lra(\beta\lra(\neg\beta\lra\alpha))&[\hbox{Axiom \ref{ax:bc1}}]\\
  2.&\circ\beta&[\circ\beta\in\Delta]\\
  3.&\beta\lra(\neg\beta\lra\alpha)&[\hbox{MP on (1) and (2)}]\\
  4.& \beta&[\beta\in\Delta]\\
  5.&\neg\beta\lra\alpha&[\hbox{MP on (3) and (4)}]\\
  6.&\neg\beta&[\neg\beta\in\Delta]\\
  7.&\alpha&[\hbox{MP on (5) and (6)}]
 \end{array}
\]
So, $\Delta\vdvd \alpha$. However, this contradicts the fact that $\Delta$ is $\alpha$-saturated set. Hence, either $\beta\notin \Delta$ or $\neg\beta\notin \Delta$.

Conversely, suppose that $\beta\notin\Delta$ or $\neg\beta\notin \Delta $. Then, by part \ref{prop:and}, $\beta\land\neg\beta\notin\Delta$. Since $\Delta$ is $\alpha$-saturated, this implies that $\Delta\cup\{\beta\land\neg\beta\}\vdvd\alpha$. So, by the deduction theorem (Theorem \ref{thm:DedThm}), $\Delta\vdvd(\beta\land\neg\beta)\lra\alpha$. Now, if possible, let $\circ\beta\notin \Delta$. Then, again as $\Delta$ is $\alpha$-saturated, $\Delta\cup\{\circ\beta\}\vdvd\alpha$. So, by the Deduction theorem, $\Delta\vdvd\circ\beta\lra\alpha$. By Lemma \ref{lem:maxrel->closed}, $\Delta$ is closed since it is $\alpha$-saturated. Thus, $(\beta\land\neg\beta)\lra\alpha,\circ\beta\lra\alpha\in \Delta$. Then, we can construct the following derivation of $\alpha$ from $\Delta$.
\[
    \begin{array}{rll}
    1.&\circ\beta\lor(\beta\land\neg\beta)&[\hbox{Axiom \ref{ax:ciw}}]\\
    2.&(\beta\land\neg\beta)\lra\alpha&[\hbox{From }\Delta]\\
    3.&\circ\beta\lra\alpha &[\hbox{From }\Delta]\\
    4.&(\circ\beta\lra\alpha)\lra(((\beta\land\neg\beta)\lra\alpha)\lra(\circ\beta\lor(\beta\land\neg\beta))\lra\alpha))&[\hbox{Axiom \ref{ax:xiv}}]\\
    5.&((\beta\land\neg\beta)\lra\alpha)\lra(\circ\beta\lor(\beta\land\neg\beta))\lra\alpha)&[\hbox{MP on (3) and (4)}]\\
    6.&\circ\beta\lor(\beta\land\neg\beta))\lra\alpha&[\hbox{MP on (2) and (5)}]\\
    7.&\alpha&[\hbox{MP on (1) and (6)}]
  \end{array}
\]
Hence, $\Delta\vdvd\alpha$, which is a contradiction, since $\Delta$ is $\alpha$-saturated. Thus, $\circ\beta\in \Delta$.
\end{enumerate}
\end{proof}

\subsection{Canonical model and completeness}

We next proceed to define a canonical topological model for $\vd=\langle\lang,\vdvd\rangle$. Let 
\[
X_{c}:=\{\Delta\subseteq\mathcal{L}\mid\,\Delta\hbox{ is $\alpha$-saturated in $\vd$ for some }\alpha\in\lang\}.
\]
For any $\varphi\in\lang$, let $F_{\varphi}:=\{\Delta\in X_{c}\mid\,\varphi\notin\Delta\}$ and $\mathcal{B}=\{F_\varphi\mid\,\varphi\in\lang\}$.

\begin{lem}\label{lem:equalsets}
    Suppose $\varphi,\psi\in\lang$. Then, the following statements hold.
    \begin{enumerate}[label=(\roman*)]
        \item\label{lem:equalsets-i} If $\varphi\equiv\psi$, i.e., $\{\varphi\}\vdvd\psi$ and $\{\psi\}\vdvd\varphi$, or equivalently by the Deduction theorem (Theorem \ref{thm:DedThm}), $\vdvd\varphi\lra\psi$ and $\vdvd\psi\lra\varphi$, then $F_{\varphi}=F_{\psi}$.
        \item\label{lem:equalsets-ii} $X_c\setminus F_\varphi=F_{\sim\varphi}$, i.e., $F_{\sim\varphi}=\{\Delta\in X_c\mid\,\varphi\in\Delta\}$.
    \end{enumerate}
\end{lem} 
  
  \begin{proof} 
  \begin{enumerate}[label=(\roman*)]
      \item Suppose $\varphi\equiv\psi$. Let $\Delta_{1}\in F_{\varphi}$. So, $\Delta_{1}$ is an $\alpha$-saturated set, for some $\alpha\in\lang$, such that $\varphi\notin\Delta_1$. If possible, let $\psi\in \Delta_{1}$. Now, since $\vdvd \psi \lra\varphi$, this implies that $\Delta_1\vdvd\psi\lra\varphi$, by monotonicity. Then, we have the following derivation of $\varphi$ from $\Delta_1$.
      \[
      \begin{array}{rll}
           1.&\psi\lra\varphi&[\Delta_1\vdvd\psi\lra\varphi]\\
           2.&\psi&[\psi\in\Delta]\\
           3.&\varphi&[\hbox{MP on (1) and (2)}]
      \end{array}
      \]
    So, $\Delta_1\vdvd\varphi$, which implies that $\varphi\in \Delta_{1}$, since $\Delta_1$ is $\alpha$-saturated, and hence, closed by Lemma \ref{lem:maxrel->closed}. This is a contradiction. Thus, $\psi\notin\Delta_1$. Hence, $\Delta_1\in F_\psi$. So, $F_\varphi\subseteq F_\psi$.
      
    By analogous reasoning, we get $F_{\psi}\subseteq F_{\varphi}$. Thus, $F_{\varphi}=F_{\psi}$.

    \item By Lemma \ref{lem:satu'ed sets prop}\ref{prop:cl_neg}, for any $\Delta\subseteq\lang$ that is $\alpha$-saturated, for some $\alpha\in\lang$, $\sim\varphi\notin\Delta$ iff $\varphi\in\Delta$. This implies that $\Delta\in F_{\sim\varphi}$ iff $\Delta\notin F_\varphi$, i.e., $F_{\sim\varphi}=X_c\setminus F_\varphi=\{\Delta\in X_c\mid\,\varphi\in\Delta\}$.
  \end{enumerate}
  
\end{proof}

\begin{lem}\label{lem:coll'B}
    Suppose $X_c$ and $\mathcal{B}$ be as described above. Then the following statements hold.
    \begin{enumerate}[label=(\roman*)]
        \item\label{lem:coll'B_empty/X} $\emptyset,X_c\in\mathcal{B}$.
        \item\label{lem:coll'B_union-and} $F_\varphi\cup F_\psi=F_{\varphi\land\psi}$, for any $\varphi,\psi\in\lang$. Thus, for any $F_1,F_2\in\mathcal{B}$, $F_1\cup F_2\in\mathcal{B}$.
    \end{enumerate}
\end{lem}

\begin{proof}
\begin{enumerate}[label=(\roman*)]
    \item Let $\varphi=\beta\lor\neg\beta$, for some $\beta\in\lang$, and $\Delta$ be an $\alpha$-saturated set, for some $\alpha\in\lang$. Then, $\varphi$ is an instance of an axiom of $\vd$ (Axiom \ref{ax:ix}), and hence, $\vdvd\varphi$. So, by monotonicity, $\Delta\vdvd\varphi$. Since $\Delta$ is $\alpha$-saturated, and hence, closed, by Lemma \ref{lem:maxrel->closed}, $\varphi\in\Delta$. Then, by Lemma \ref{lem:satu'ed sets prop}\ref{prop:cl_neg}, $\sim\varphi\notin\Delta$. 

    This proves that there is no $\alpha$-saturated set that does not contain $\varphi$, and hence, $F_\varphi=F_{\beta\lor\neg\beta}=\emptyset$. Moreover, no $\alpha$-saturated set contains $\sim\varphi$, and so, for any $\alpha$-saturated $\Delta$, $\sim\varphi\notin\Delta$, and thus, $F_{\sim\varphi}=F_{\sim(\beta\lor\neg\beta)}=X_c$. Hence, $\emptyset,X_c\in\mathcal{B}$.

    \item Let $\varphi,\psi\in\lang$ and $\Delta\in F_{\varphi}$. Then, $\Delta$ is $\alpha$-saturated, for some $\alpha\in\lang$, and $\varphi\notin\Delta$. Since $\Delta$ is $\alpha$-saturated, by Lemma \ref{lem:maxrel->closed}, it is closed. So, $\varphi\notin\Delta$ implies that $\Delta\not\vdvd\varphi$. If possible, let $\varphi\land\psi\in\Delta$. Then, we have the following derivation of $\varphi$ from $\Delta$.
    \[
    \begin{array}{rll}
         1.&(\varphi\land\psi)\lra\varphi&[\hbox{Axiom \ref{ax:iv}}]\\
         2.&\varphi\land\psi&[\varphi\land\psi\in\Delta\hbox{ by assumption}]\\
         3.&\varphi&[\hbox{MP on (1) and (2)}]
    \end{array}
    \]
    So, $\Delta\vdvd\varphi$. This, however, contradicts our previous conclusion. Thus, $\varphi\land\psi\notin\Delta$, which implies that $\Delta\in F_{\varphi\land\psi}$. Hence, $F_\varphi\subseteq F_{\varphi\land\psi}$. By analogous arguments, $F_{\psi}\subseteq F_{\varphi\land\psi}$. So, $F_\varphi\cup F_\psi\subseteq F_{\varphi\land\psi}$.

    Now, let $\Delta\in F_{\varphi\land\psi}$. Then, $\Delta$ is $\alpha$-saturated, for some $\alpha\in\lang$, and $\varphi\land\psi\notin\Delta$. Using the same reasoning as before, $\Delta$ is closed, and hence, $\Delta\not\vdvd\varphi\land\psi$. If possible, let $\varphi,\psi\in\Delta$. Then, we have the following derivation of $\varphi\land\psi$ from $\Delta$.
    \[
    \begin{array}{rll}
         1.&\varphi\lra(\psi\lra(\varphi\land\psi))&[\hbox{Axiom \ref{ax:iii}}]\\
         2.&\varphi&[\varphi\in\Delta]\\
         3.&\psi\lra(\varphi\land\psi)&[\hbox{MP on (1) and (2)}]\\
         4.&\psi&[\psi\in\Delta]\\
         5.&\varphi\land\psi&[\hbox{MP on (3) and (4)}]
    \end{array}
    \]
    So, $\Delta\vdvd\varphi\land\psi$, which is a contradiction. Thus, either $\varphi\notin\Delta$ or $\psi\notin\Delta$. This implies that either $\Delta\in F_\varphi$ or $\Delta\in F_\psi$, i.e., $\Delta\in F_\varphi\cup F_\psi$. So, $F_{\varphi\land\psi}\subseteq F_\varphi\cup F_\psi$. Hence, $F_{\varphi\land\psi}=F_{\varphi}\cup  F_{\psi}$.
\end{enumerate}
\end{proof}

\begin{lem}\label{lem:kuratowski-like}
    Suppose $X_c$ and $\mathcal{B}$ be as described above. Then, the map $\wh{(.)}:\mathcal{B}\lra\mathcal{B}$, defined by
    \[
    \wh{F_{\varphi}}=F_{\sim\neg\varphi}=\{\Delta\in X_{c}\mid\,\neg\varphi\in\Delta\},\;\hbox{ for any }F_\varphi\in\mathcal{B},
    \]
     is a Kuratowski-like operator on $\mathcal{B}$.
\end{lem}

\begin{proof}
    To show that $\wh{(\cdot)}$ is a Kuratowski-like operator on $\mathcal{B}$, we verify the conditions in Definition \ref{def:kuratowski-like} below.
    \begin{enumerate}
        \item We know from the proof of Lemma \ref{lem:coll'B}\ref{lem:coll'B_empty/X} that $\emptyset=F_{\varphi\lor\neg\varphi}$, where $\varphi\in\lang$. Then,
        \[
        \wh{\emptyset}=\wh{F_{\varphi\lor\neg\varphi}}=F_{\sim\neg(\varphi\lor\neg\varphi)}=\{\Delta\in X_c\mid\,\neg(\varphi\lor\neg\varphi)\in\Delta\}.
        \]
        If possible, let $\wh{\emptyset}\neq\emptyset$. Suppose $\Delta\in\wh{\emptyset}$. Then, $\Delta$ is $\alpha$-saturated, for some $\alpha\in\lang$, and $\neg(\varphi\lor\neg\varphi)\in\Delta$. Now, by Lemma \ref{lem:satu'ed sets prop}\ref{prop:ax in sat'set}, $\sim\neg(\varphi\lor\neg\varphi)\in\Delta$. Then, by Lemma \ref{lem:satu'ed sets prop}\ref{prop:cl_neg}, $\neg(\varphi\lor\neg\varphi)\notin\Delta$. This is a contradiction. Thus, $\wh{\emptyset}=\emptyset$.

        \item Let $F_\varphi\in\mathcal{B}$, where $\varphi\in\lang$, and $\Delta\in F_{\varphi}$. Then, $\Delta$ is $\alpha$-saturated, for some $\alpha\in\lang$, and $\varphi\notin\Delta$. So, by Lemma \ref{lem:satu'ed sets prop}\ref{prop:pc neg}, $\neg\varphi\in\Delta$. This implies that $\Delta\in\wh{F_\varphi}$. Hence, $F_\varphi\subseteq\wh{F_\varphi}$.
        
        \item Let $F_\varphi,F_\psi\in\mathcal{B}$, where $\varphi,\psi\in\lang$. Then, $\wh{F_\varphi}\cup\wh{F_\psi}=F_{\sim\neg\varphi}\cup F_{\sim\neg\psi}=F_{\sim\neg\varphi\land\sim\neg\psi}$, by Lemma \ref{lem:coll'B}\ref{lem:coll'B_union-and}. Now, by Axioms \ref{ax:xvii} and \ref{ax:viii},
        \[
        \begin{array}{c}
             \vdvd(\sim\neg(\varphi\land\psi))\lra(\sim\neg\varphi\land\sim\neg\psi)\\
             \hbox{and}\\
             \vdvd(\sim\neg\varphi\land\sim\neg\psi)\lra(\sim\neg(\varphi\land\psi)).
        \end{array}
        \]
        So, $\sim\neg(\varphi\land\psi)\equiv\sim\neg\varphi\land\sim\neg\psi$, and thus, by Lemma \ref{lem:equalsets}\ref{lem:equalsets-i}, $F_{\sim\neg\varphi\land\sim\neg\psi}=F_{\sim\neg(\varphi\land\psi)}=\wh{F_{\varphi\land\psi}}$. Since $F_{\varphi\land\psi}=F_\varphi\cup F_\psi$, by Lemma \ref{lem:coll'B}\ref{lem:coll'B_union-and}, $\wh{F_{\varphi\land\psi}}=\wh{F_\varphi\cup F_\psi}$. Hence, $\wh{F_\varphi}\cup\wh{F_\psi}=\wh{F_\varphi\cup F_\psi}$.

        \item Let $F_\varphi\in\mathcal{B}$, where $\varphi\in\lang$. Then, $\wh{\wh{F_\varphi}}=\wh{F_{\sim\neg\varphi}}=F_{\sim\neg\sim\neg\varphi}$. Now, by Axioms \ref{ax:xv} and \ref{ax:xvi},
        \[
        \begin{array}{c}
             \vdvd\sim\neg\varphi\lra\sim\neg\sim\neg\varphi\\
             \hbox{and}\\
             \vdvd\sim\neg\sim\neg\varphi\lra\sim\neg\varphi.
        \end{array}
        \]
        So, $\sim\neg\varphi\equiv\sim\neg\sim\neg\varphi$, and thus, by Lemma \ref{lem:equalsets}\ref{lem:equalsets-i}, $F_{\sim\neg\sim\neg\varphi}=F_{\sim\neg\varphi}=\wh{F_\varphi}$. Hence, $\wh{\wh{F_\varphi}}=\wh{F_\varphi}$.
    \end{enumerate}
    Thus, $\wh{(\cdot)}$ is a Kuratowski-like operator on $\mathcal{B}$.
\end{proof}

\begin{cor}\label{cor:Kuratwoski}
    Suppose $X_c,\mathcal{B}$ and $\wh{(\cdot)}:\mathcal{B}\to\mathcal{B}$ be as in the above lemma. Then, the operator $\ov{(\cdot)}:\pow(X_c)\to\pow(X_c)$, defined by 
    \[
    \overline{A}=\bigcap\left\{\wh{F}\mid\,F\in\mathcal{B}\hbox{ and }A\subseteq\wh{F}\right\}=\bigcap\left\{\wh{F_\varphi}\mid\,\varphi\in\lang,\, A\subseteq \wh{F_\varphi}\right\}=\bigcap\left\{F_{\sim\neg\varphi}\mid\,\varphi\in\lang,\, A\subseteq F_{\sim\neg\varphi}\right\},
    \]
    for any $A\subseteq X_c$, is a Kuratowski closure operator over $X_c$. Moreover, for any $F_\varphi\in\mathcal{B}$, $\ov{F_\varphi}=\wh{F_\varphi}$.
\end{cor}

\begin{proof}
    This follows from Theorem \ref{thrm:kura'ki result} and Lemma \ref{lem:kuratowski-like}.
\end{proof}

Now, since every Kuratowski closure operator over a set $X$ generates a unique topology on $X$, the Kuratowski closure operator over $X_c$, described above, generates a unique topology on $X_c$. Let $\tau_c$ be this topology and $\mathcal{M}_c=\langle X_c,\tau_c\rangle$. We now define a valuation for $\vd=\langle\lang,\vdvd\rangle$ into the topological space $\mathcal{M}_c$.

\begin{lem}\label{lem:vc val'on for vd}      
    Suppose $v_{c}:\lang\to\pow(X_{c})$ is the map defined as follows. For any $\varphi\in\lang$, $v_{c}(\varphi)=F_{\sim\varphi}$. Then, $v_c$ is a valuation for $\vd$.    
\end{lem}

\begin{proof}
We first note that, by Lemma \ref{lem:equalsets}\ref{lem:equalsets-ii}, 
\[
v_{c}(\varphi)=F_{\sim\varphi}=X_c\setminus F_\varphi=\{\Delta\in X_{c}\mid\,\varphi\in \Delta\}.
\]
Now, to show that $v_c$ is a valuation for $\vd$, we verify the conditions in Definition \ref{def:val'on} below. 
\begin{enumerate}[label=(\arabic*)]
    \item Suppose $\varphi,\psi\in\lang$. Then, $v_c(\varphi\lra\psi)=F_{\sim(\varphi\lra\psi)}=\{\Delta\in X_{c}\mid\,\varphi\lra\psi\in \Delta\}$. Now, by Lemma \ref{lem:satu'ed sets prop}\ref{prop:arrow}, for any $\alpha$-saturated set $\Delta$, where $\alpha\in\lang$, $\varphi\lra\psi\in \Delta$ iff, either $\varphi \notin \Delta$ or $\psi\in \Delta$. Since $\varphi\notin\Delta$ implies that $\Delta\in F_\varphi$ and $\psi\in\Delta$ implies that $\Delta\in F_{\sim\psi}$, $F_{\sim(\varphi\lra\psi)}=F_\varphi\cup F_{\sim\psi}$. By Lemma \ref{lem:equalsets}\ref{lem:equalsets-ii}, $F_\varphi=X_c\setminus F_{\sim\varphi}$. So, 
    \[
    v_c(\varphi\lra\psi)=F_{\sim(\varphi\lra\psi)}=(X_c\setminus F_{\sim\varphi})\cup F_{\sim\psi}=v_c(\varphi)^c\cup v_c(\psi).
    \]

    \item Suppose $\varphi\in\lang$. Then, $v_c(\neg\varphi)=F_{\sim\neg\varphi}$. Now, 
    \[
    \begin{array}{rcll}
         \ov{v_c(\varphi)^c}&=&\ov{X\setminus F_{\sim\varphi}}&\\
         &=&\ov{F_\varphi}&[\hbox{By Lemma \ref{lem:equalsets}\ref{lem:equalsets-ii}}]\\
         &=&\wh{F_{\varphi}}&[\hbox{by Corollary \ref{cor:Kuratwoski}, since }F_{\varphi}\in\mathcal{B}]
    \end{array}
    \]
     Then, by the definition of $\wh{(\cdot)}$ given in Lemma \ref{lem:kuratowski-like}, $\wh{F_{\varphi}}=F_{\sim\neg\varphi}$. Thus,
    \[
    v_c(\neg\varphi)=F_{\sim\neg\varphi}=\ov{v_c(\varphi)^c}.
    \]

    \item Suppose $\varphi\in\lang$. Then, $v_{c}(\sim\varphi)=F_{\sim\sim\varphi}=X_c\setminus F_{\sim\varphi}=X_c\setminus v_c(\varphi)=v_c(\varphi)^c$.
    
    \item Suppose $\varphi,\psi\in\lang$. Then, $v_c(\varphi\land\psi)=F_{\sim(\varphi\land\psi)}=\{\Delta\in X_{c}\mid\,\varphi\land\psi\in \Delta\}$. Now, by Lemma \ref{lem:satu'ed sets prop}\ref{prop:and}, for any $\alpha$-saturated set $\Delta$, where $\alpha\in\lang$, $\varphi\land\psi\in \Delta$ iff, $\varphi\in\Delta$ and $\psi\in \Delta$. Then, as $\varphi\in\Delta$ implies that $\Delta\in F_{\sim\varphi}$ and $\psi\in\Delta$ implies that $\Delta\in F_{\sim\psi}$, 
    \[    v_c(\varphi\land\psi)=F_{\sim(\varphi\land\psi)}=F_{\sim\varphi}\cap F_{\sim\psi}=v_c(\varphi)\cap v_c(\psi).
    \]

    \item Suppose $\varphi,\psi\in\lang$. Then, $v_c(\varphi\lor\psi)=F_{\sim(\varphi\lor\psi)}=\{\Delta\in X_{c}\mid\,\varphi\lor\psi\in \Delta\}$. By Lemma \ref{lem:satu'ed sets prop}\ref{prop:or}, for any $\alpha$-saturated set $\Delta$, where $\alpha\in\lang$, if $\varphi\in\Delta$ or $\psi\in\Delta$, then $\varphi\lor\psi\in\Delta$. Now, $\varphi\in\Delta$ implies that $\Delta\in F_{\sim\varphi}$ and $\psi\in\Delta$ implies that $\Delta\in F_{\sim\psi}$. So,
    \[
    v_c(\varphi)\cup v_c(\psi)=F_{\sim\varphi}\cup F_{\sim\psi}\subseteq F_{\sim(\varphi\lor\psi)}=v_c(\varphi\lor\psi).
    \]

    \item Suppose $\varphi\in\lang$. Then, $v_c(\circ\varphi)=F_{\sim\circ\varphi}=\{\Delta\in X_c\mid\,\circ\varphi\in\Delta\}$. Let $\Delta$ be an $\alpha$-saturated set, for some $\alpha\in\lang$. Now,
    \[
    \begin{array}{rcll}
         \circ\varphi\in\Delta&\hbox{iff}&\varphi\notin\Delta\hbox{ or }\neg\varphi\notin\Delta&[\hbox{By Lemma \ref{lem:satu'ed sets prop}\ref{prop:circ}}]\\
         &\hbox{iff}&\Delta\in F_\varphi\hbox{ or }\Delta\in F_{\neg\varphi}&\\
         &\hbox{iff}&\Delta\in X_c\setminus F_{\sim\varphi}\hbox{ or }\Delta\in X_c\setminus F_{\sim\neg\varphi}&[\hbox{By Lemma \ref{lem:equalsets}\ref{lem:equalsets-ii}}]\\
         &\hbox{iff}&\Delta\in X_c\setminus v_c(\varphi)\hbox{ or }\Delta\in X_c\setminus v_c(\neg\varphi)&\\
         &\hbox{iff}&\Delta\in X_c\setminus v_c(\varphi)\hbox{ or }\Delta\in X_c\setminus \ov{v_c(\varphi)^c}&[\hbox{By part (2) above}]\\
         &\hbox{iff}&\Delta\in v_c(\varphi)^c\hbox{ or }\Delta\in \left(\ov{v_c(\varphi)^c}\right)^c&\\
         &\hbox{iff}&\Delta\in v_c(\varphi)^c\hbox{ or }\Delta\in \Int(v_c(\varphi))&[\hbox{By Lemma \ref{lem:compclcomp}}]\\
         &\hbox{iff}&\Delta\in v_c(\varphi)^c\cup\Int(v_c(\varphi))&\\
    \end{array}
    \]
    Thus, 
    \[
    v_c(\circ\varphi)=\{\Delta\in X_c\mid\,\circ\varphi\in\Delta\}=v_c(\varphi)^c\cup\Int(v_c(\varphi)).
    \]
\end{enumerate}
Hence, $v_c$ is a valuation for $\vd$.
\end{proof}

\begin{dfn}
    The topological space $\mathcal{M}_{c}=\langle X_{c},\tau_{c}\rangle$, where $X_c$ is as described earlier, and $\tau_{c}$ is the unique topology over $X_{c}$, generated by the Kuratowski closure operator $\ov{(\cdot)}$ defined above, is called the \emph{canonical structure} for $\vd$. The topological model $\langle\mathcal{M}_{c},v_{c}\rangle$, where $v_{c}(\varphi)=F_{\sim\varphi}=\{\Delta\in X_{c}\mid\,\varphi\in \Delta\}$, for any $\varphi\in\lang$, is called the \emph{canonical model} for $\vd$. 
\end{dfn}

\begin{thm}[\textsc{Completeness}]
    For any $\Gamma\cup\{\varphi\}\subseteq\lang$, if $\Gamma\modvd\varphi$, then $\Gamma\vdvd\varphi$.
\end{thm}
\begin{proof} 
    Suppose $\Gamma\modvd\varphi$. 

    \textsc{Case 1:} $\modvd\varphi$. Then, $v(\varphi)=X$, for every topological model $\langle\langle X,\tau\rangle,v\rangle$ for $\vd$. So, in particular, in the canonical model for $\vd$, $v_c(\varphi)=X_c$. Now, $v_c(\varphi)=F_{\sim\varphi}=\{\Delta\in X_c\mid\,\varphi\in\Delta\}$. Thus, $\{\Delta\in X_c\mid\,\varphi\in\Delta\}=X_c$, i.e., $\varphi\in\Delta$ for every set $\Delta\subseteq\lang$ that is $\alpha$-saturated, for some $\alpha\in\lang$, in $\vd$.

    If possible, let $\not\vdvd\varphi$. Then, by Theorem \ref{thm:Linde'os}, there exists a $\Delta^\prime\supseteq\emptyset$ that is $\varphi$-saturated in $\vd$. So, $\Delta^\prime\not\vdvd\varphi$. Then, by reflexivity, $\varphi\notin\Delta^\prime$. This is in contradiction with our previous conclusion. Hence, $\vdvd\varphi$, and so, by monotonicity, $\Gamma\vdvd\varphi$.

    \textsc{Case 2:} There exists a non-empty, finite $\Gamma_0\subseteq\Gamma$ such that $\displaystyle\bigcap_{\gamma\in\Gamma_0}v(\gamma)\subseteq v(\varphi)$, for every topological model $\langle\langle X,\tau\rangle,v\rangle$ for $\vd$. So, in particular, using the canonical model for $\vd$, $\displaystyle\bigcap_{\gamma\in\Gamma_0}v_c(\gamma)\subseteq v_c(\varphi)$.

    If possible, let $\Gamma\not\vdvd\varphi$ and $\Gamma_0$ be a finite subset of $\Gamma$. Then, $\Gamma_0\not\vdvd\varphi$. So, by Theorem \ref{thm:Linde'os}, there exists a $\varphi$-saturated $\Delta^\prime\supseteq\Gamma_0$. This implies that $\Delta^\prime\not\vdvd\varphi$, and hence, by reflexivity, $\varphi\notin\Delta^\prime$. Now, $v_c(\varphi)=\{\Delta\in X_c\mid\,\varphi\in\Delta\}$. Thus, $\Delta^\prime\notin v_c(\varphi)$.

    Now, for any $\gamma\in\Gamma_0$, $v_c(\gamma)=\{\Delta\in X_c\mid\,\gamma\in\Delta\}$. Then, since $\Gamma_0\subseteq\Delta^\prime$, $\gamma\in\Delta^\prime$ for all $\gamma\in\Gamma_0$, and hence, $\Delta^\prime\in v_c(\gamma)$ for all $\gamma\in\Gamma_0$. Thus, $\Delta^\prime\in\displaystyle\bigcap_{\gamma\in\Gamma_0}v_c(\gamma)$. So, $\displaystyle\bigcap_{\gamma\in\Gamma_0}v_c(\gamma)\not\subseteq v_c(\varphi)$.

    Since $\Gamma_0$ was an arbitrary finite subset of $\Gamma$, this implies that there exists no finite $\Gamma_0\subseteq\Gamma$ such that $\displaystyle\bigcap_{\gamma\in\Gamma_0}v_c(\gamma)\subseteq v_c(\varphi)$. This is a contradiction. Thus, $\Gamma\vdvd\varphi$.
\end{proof}

\section{Conclusion}
In this article, we have introduced an \textbf{LFI}, that we have named $\vd$, which is not self-extensional. We have proved that $\vd$ is sound and complete with respect to the topological semantics given in Section \ref{sec:top}.   
We hope to further extend the ideas to the Logics of Formal Undeterminedness (LFUs) and Logics of Formal Inconsistency and Undeterminedness (LFIUs) that are not self-extensional. One could also study the relations between different such topological models. These are left for future work.

  \bibliography{TSLFI}
  \bibliographystyle{plain}
   \end{document}